\documentclass[11pt]{amsart}
\usepackage{tikz}
\usepackage{mathrsfs}
\usepackage[mathcal]{euscript}
\usepackage{graphicx}
\usepackage{amsthm,amscd}
\usepackage{calc}
\usepackage{color}

\usepackage{mathrsfs,dsfont}
\usepackage{fancyhdr,amscd}
\usepackage{anysize}
\usepackage{amsmath,amssymb,amsfonts}
\usepackage{epsfig,enumerate}
\usepackage{latexsym}
\usepackage{indentfirst, latexsym}
\usepackage{graphics}
\usepackage[all,poly,knot]{xy}
\usepackage[pagebackref]{hyperref}%此命令使得每条参考文献后面用红框显示引用该条文献的在哪一页，并且点击红框会跳到引用的具体位置，即有超链接功能。

\newcommand{\comment}[1]{}

\usepackage{fancyhdr}
\providecommand{\U}[1]{\protect\rule{.1in}{.1in}}

\numberwithin{equation}{section}

\def\Div{{\rm Div}}

\def\Bl{{\rm Bl}}
\def\rk{{\rm rk}}

\theoremstyle{plain}
\newtheorem{thm}{Theorem}[section]
\newtheorem{lemma}[thm]{Lemma}

\newtheorem{prop}[thm]{Proposition}
\newtheorem{cor}[thm]{Corollary}

\newtheorem{con}[thm]{Conjecture}

\theoremstyle{definition}
\newtheorem{defn}[thm]{Definition}
\newtheorem{ex}[thm]{Example}
\newtheorem{rem}[thm]{Remark}
\newtheorem{claim}[thm]{Claim}

\newtheorem{observation}[thm]{Observation}

\newtheorem{step}[]{Step}
\newtheorem{notation}[thm]{Notation}

\begin{document}

%\linenumbers

\title[Characterization of fiberwise bimeromorphism and specialization]{Characterization of fiberwise bimeromorphism and specialization of bimeromorphic types I: locally Moishezon case}

\author[Jian Chen]{Jian Chen}
\author[Sheng Rao]{Sheng Rao}
\author[I-Hsun Tsai]{I-Hsun Tsai}
	
\address{Jian Chen, School of Mathematics and Statistics, Central China Normal University,
		Wuhan 430079, People's Republic of China}
\email{jian-chen@whu.edu.cn}
	
\address{Sheng Rao, School of Mathematics and Statistics, Wuhan  University,
		Wuhan 430072, People's Republic of China}
\email{likeanyone@whu.edu.cn}
	
\address{I-Hsun Tsai, Department of Mathematics, National Taiwan University, Taipei 10617,
		Taiwan}
\email{ihtsai@math.ntu.edu.tw}

\thanks{The first two authors are partially supported by NSFC (Grant No. 12271412, W2441003) and Hubei Provincial Innovation Research Group Project (Grant No. 2025AFA044). The third author would like to thank Taiwan's Ministry of Education for financial support.}
	
\date{\today}
	
\begin{abstract}
Inspired by the recent works of M. Kontsevich--Y. Tschinkel and J. Nicaise--J. C. Ottem on specialization of birational types for smooth families (in the scheme category) and J. Koll{\'a}r's work on fiberwise bimeromorphism, we focus on characterizing the fiberwise bimeromorphism and utilizing the characterization to investigate the specialization of bimeromorphic types for non-smooth families in the complex analytic setting.  We provide several criteria for a bimeromorphic map between two families over the same base to be fiberwise bimeromorphic. By combining these criteria with the relative Barlet cycle space theoretic argument motivated by D. Mumford--U. Persson, K. Timmerscheidt and T. de Fernex--D. Fusi, we establish the specialization of bimeromorphic types for locally Moishezon families with fibers having only canonical singularities and being of non-negative Kodaira dimension.  These specialization results can easily lead to criteria for locally strongly bimeromorphic isotriviality. Throughout this paper, we unveil the connections among the four classical topics in bimeromorphic geometry: the deformation behavior of plurigenera (or even $1$-genus), fiberwise bimeromorphism, specialization of bimeromorphic types, and the bimeromorphic version of the deformation rigidity.	
\end{abstract}

\maketitle
	
\setcounter{tocdepth}{1}%这里设置pdf中显示的目录深度层次，让\subsection 的标题也显示出来
\tableofcontents
	
\section{Introduction}\label{s introduction}
	
The specialization of rationality and other related concepts (such as stable rationality, unirationality, uniruledness, etc.) are classical and significant topics of algebraic geometry, rapidly developing in recent years. Among them is M. Kontsevich--Y. Tschinkel's remarkable work \cite[Theorems 1 and 16]{KT19}, which fully addresses the specialization of rationality for smooth families (or the pair $(\text{total space}, \text{special fiber})$  has $B$-rational singularities) with arbitrary-dimensional fibers, in the scheme category.   In fact, they have achieved even more on the specialization of birational types in the case of smooth families, as shown in Theorem \ref{introduction-KT}. Note that these specialization questions can be regarded as the bimeromorphic counterpart to the non-deformability (rigidity) questions, which have been widely studied by Y.-T. Siu \cite{Si89,Si92}, J.-M. Hwang--N. Mok \cite{HM98, HM02, HM05}, Q. Li \cite{L22},  Y. Chen--B. Fu--Q. Li \cite{CFL23},  M. Li--X. Liu \cite{LL24}, and others.  
\begin{thm}[{\cite[Theorem 1]{KT19}}]\label{introduction-KT}
Let $\pi: X \rightarrow B$ and $\pi^{\prime}: X^{\prime} \rightarrow B$ be smooth proper morphisms to a smooth connected curve $B$ over a field of characteristic zero.  Suppose that the generic fibers of $\pi$ and $\pi^{\prime}$ are birational over the function field of $B$, then for every closed point $b \in B$ the fibers of $\pi$ and $\pi^{\prime}$ over $b \in B$ are birational over the residue field at $b$.	In particular, if the generic fiber of $\pi$ is rational,  then every fiber of $\pi$ is rational. 
\end{thm}
	
Then J. Nicaise--J. C. Ottem \cite[Theorem 4.1.1]{NO21} studied the specialization of geometric (stable) birational types and described the structure of the geometric (stable) birational locus for proper smooth families in the scheme category. 
	
\begin{thm}[{\cite[Theorem 4.1.1]{NO21}}]\label{intro-NO}
Let $S$ be a Noetherian scheme of characteristic zero, and let $\mathscr{X} \rightarrow$ $S$ and $\mathscr{Y} \rightarrow S$ be smooth and proper $S$-schemes. For every point $s$ of $S$, we fix a geometric point $\bar{s}$ supported at $s$. Then the subset 
$$
\begin{aligned}
S_{\mathrm{bir}}(\mathscr{X}, \mathscr{Y})  =\left\{s \in S \mid \mathscr{X} \times_S \bar{s}\sim_{\mathrm{bir}} \mathscr{Y} \times_S \bar{s}\right\}
\end{aligned}
$$
is countable unions of closed subsets of $S$. 
\end{thm}
	
Loosely speaking, the approaches of both Theorems \ref{introduction-KT} and \ref{intro-NO} are to construct some specialization map of certain algebraic structure associated with the function field to that associated with the residue field at the closed point of the base space. However, it seems difficult to construct an analogous specialization morphism in the complex analytic setting.  It is natural to seek an alternative method, different from \cite{KT19}, to address the specialization question of bimeromorphic types in the complex analytic setting. This serves as the primary motivation for the present work. Generally speaking, by \emph{specialization (or degeneration) question of bimeromorphic types}, we refer to: Given two deformation families over the same base space, with fibers (over a general point of the base) that are bimeromorphically equivalent, under what conditions are all the fibers over a special point of the base bimeromorphically equivalent?
	
Recently, J. Koll{\'a}r  studied a fiberwise bimeromorphism as in \cite[Definition 26]{Kol22} (or  Definition \ref{fiberbimedef}) in the complex analytic category, possibly initiated in \cite{KMo92}. He then formulated a profound \cite[Conjecture 5]{Kol22} (or just Conjecture \ref{conj5}) and presented an intriguing result on a fiberwise bimeromorphism in \cite[Theorem 28]{Kol22}.  
	
In this paper, we aim to address the specialization question of bimeromorphic types in the complex analytic setting. This is done by constructing a global bimeromorphic map between the total spaces of the two families and then demonstrating that this map is fiberwise bimeromorphic (when each fiber of both families under consideration is irreducible).
	
As is well known, bimeromorphic transformations of the total space of a family are common operations in bimeromorphic geometry. So, other than playing an important role in the study of specialization problems,  investigating the conditions under which a bimeromorphic map between two families over the same base is fiberwise bimeromorphic is also a natural and fundamental question, with its own independent significance and interest.
	
In addressing the specialization question of bimeromorphic types, we first examine the conditions under which a bimeromorphic map between two families is fiberwise bimeromorphic, as defined in Definition \ref{fiberbimedef}. We then apply this result, along with the relative Barlet cycle space theory, to study the specialization (also known as degeneration) of bimeromorphic types, as well as the specialization of certain fixed bimeromorphic structures for locally Moishezon families. 
	
In this paper, we obtain several fiberwise bimeromorphism results for families whose fibers may be reducible, and several specialization results of bimeromorphic types for singular families. Furthermore,  we explore in this paper the relationships among four classical topics in bimeromorphic geometry: the deformation behavior of plurigenera (or even $1$-genus) implies fiberwise bimeromorphism, which in turn implies the specialization of a fixed bimeromorphic structure. Moreover, the specialization of a fixed bimeromorphic structure is equivalent, at the methodological level, to the specialization of bimeromorphic types.
	
We emphasize that, in contrast to \cite[Theorems 1, 16]{KT19}, some of our main theorems in this paper do not require the families involved to be smooth, and they allow for reducible fibers. Furthermore, the total spaces, base spaces, and fibers are not required to be algebraic (i.e., complex analytic spaces associated with schemes). More specifically, in Sections \ref{section-moishezon} and \ref{s defor limit}, we focus on families whose fibers are Moishezon (which correspond to algebraic spaces).  {Note that we also establish several results on the fiberwise bimeromorphism and  specialization of bimeromorphic types for certain non-smooth locally Fujiki families, under different settings (\cite{CRT26}).} Additionally, in all the results of this paper, each bimeromorphism between fibers that we obtain is induced by  a global bimeromorphic map (e.g., Definition \ref{fiberbimedef}) between the total spaces of the two families (possibly after shrinking the base or taking a base change). 
It's worth noting that for any two families  $f: X\to S$ and $g: Y\to S$, the condition that 
$(X_t)_{\rm{red}}$ is bimeromorphic to $(Y_t)_{\rm{red}}$ for any $t\in S$ is quite different from  $f$ being fiberwise bimeromorphic to $g$ (e.g., Example \ref{Hirze example}).   {Regrettably,  as  illustrated by Example \ref{Hirze example}, our method to  study the specialization of bimeromorphic types via fiberwise bimeromorphisms cannot, in general,  handle  families whose fibers have  Kodaira dimension $-\infty$.}   

\subsection{Conventions for concepts and notations}\label{ss-1.1}
	
Before stating the main results in this paper, we first fix the necessary notions and notations. Unless otherwise stated, throughout this paper, all complex analytic spaces are assumed to be $\mathbb{C}$-ringed spaces whose underlying topological space is Hausdorff, and which are glued together by local models (e.g., \cite{Fs76}). 
Except for the term ``irreducible", other topological notions are with respect to the Hausdorff complex topology. The definitions of proper modifications and meromorphic maps are based on Remmert's definitions (see Section \ref{ss prelima} for details). Denote the reduction of a complex analytic space $X$ by $X_{\textrm{red}}$. We use the terms ``morphism" and ``holomorphic map" interchangeably, and we often denote the finite base change of a complex analytic space (say $\mathcal{U}$) by this notation itself  (i.e., $\mathcal{U}$)  when using the semi-stable reduction   theorem (e.g., Footnote \ref{localversion-basechange}).

{A \emph{fiber} of a morphism $f: X \to Y$ over $y \in Y$,  denoted by $X_y:=f^{-1}(y)$, will always mean the  standard analytic fiber,
 i.e., the analytic subspace of $X$ defined by the ideal   sheaf $\operatorname{Im}\left(f^*\left(m_y\right) \rightarrow \mathcal{O}_X\right)$. In particular, the smooth locus of a fiber may be empty.}
A \emph{family} will refer to a proper surjective morphism from a {normal} complex analytic space to a smooth connected curve \footnote{In many occasions, we can reduce questions related to families to cases where the base is a curve. We subsequently normalize the base and consider the corresponding base change. Therefore, we always assume that the base of a family is a smooth curve.} (not necessarily compact) with connected (but possibly reducible) complex analytic spaces as its fibers.   We also assume that the general fiber of a family is irreducible.   The ``general fiber" of a family refers to the one over  any point of a nonempty analytic Zariski open subset of the base.  \footnote{{Note that, by \cite[Lemma 1.4]{Fu78/79} and \cite[Proposition 3]{Fu82}, the assumption that the general fiber of a family is irreducible is not unduly restrictive. }}  It follows from Lemmata \ref{connectedness of total space} and \ref{flatovercurve} that,    in the present paper, a family  is flat and the total space of a family is irreducible.  
 Note that the fibers of a family are compact and thus automatically countable at infinity; the total space of a family is also automatically countable at infinity by the properness condition and the fact that a smooth curve is either (compact) projective or Stein. 
	
Particularly in Sections \ref{section-fib bime for map kodaira} and \ref{section-moishezon}, we make frequent use of:
\begin{notation}\label{notation 1}
Let  
$$\pi_1: \mathcal{X} \rightarrow S \text{ and } \pi_2: \mathcal{Y} \rightarrow S$$
be two families (in particular, both $\mathcal{X}$ and $\mathcal{Y}$ are normal).   For any $t\in S$, we denote by 
$X_t:=\pi_1^{-1}(t)$ and  $Y_t:=\pi_2^{-1}(t)$  the fibers of $\pi_1$ and $\pi_2$, respectively. Let $f:\mathcal{X} \dashrightarrow \mathcal{Y}$ be a bimeromorphic map over $S$ (i.e., $\pi_1=\pi_2\circ f$) \footnote{In the remainder of this paper, we use the term ``over $S$" to denote the condition $\pi_1=\pi_2 \circ f$. However, we occasionally use the phrase ``over $S$ such that $\pi_1=\pi_2 \circ f$" interchangeably.}.		
Note that both $\mathcal{X}$ and $\mathcal{Y}$ are (globally) irreducible by Lemma \ref{connectedness of total space}.   
\end{notation}	

\subsection{Summary of the main results on fiberwise bimeromorphism}
We begin with a preliminary analysis to explore the challenges in determining when a general bimeromorphic map between two families over the same base is fiberwise bimeromorphic. Consider a given family $\mathcal{X} \to S$ with reduced and irreducible fibers. By blowing up $\mathcal{X}$ along a suitable center, we obtain a new family $\mathcal{Y} \to S$ 
such that the special fiber $Y_s$ contains both the strict transform $\mu^{-1}_*X_s$ of $X_s$ and an exceptional divisor $E$, where the blow up is denoted by $\mu: \mathcal{Y} \to \mathcal{X}$.
If it is possible to blow down $\mathcal{Y}$  to contract $\mu^{-1}_*X_s$ via a bimeromorphic map $\mathcal{Z} \dashrightarrow \mathcal{Y}$, then one produces another family $\mathcal{Z} \to S$, where the fiber $Z_s$ over $s$ corresponds to the strict transform of $E$.
	
A notable complication arises: $Z_s$ may fail to be bimeromorphic to $X_s$. To address this, it is natural to analyze two primary obstacles that may prevent a general bimeromorphic map between two families over the same base from being fiberwise bimeromorphic: the ramification divisor and the indeterminacy locus.
	
An observation regarding fiberwise bimeromorphism is as follows. Broadly speaking, it demonstrates that the ramification divisor does not pose an obstacle for a bimeromorphic morphism (not just a bimeromorphic map) between two families over the same base to be fiberwise bimeromorphic:

\begin{thm}[=Theorem \ref{new-fiberbime-irredu-singular}+Corollary \ref{morphism-cor-25}]\label{intro-fiberbime-irredu-singular}
Let $\mathcal{X}$, $\mathcal{Y}$ and $S$ be complex analytic spaces.
Assume that $\mathcal{X}$ is reduced (not necessarily normal) and irreducible, $\mathcal{Y}$  is normal and irreducible, and $S$ is irreducible (not necessarily $1$-dimensional).	
Assume further that both $\pi_1: \mathcal{X} \to  S$ and $\pi_2: \mathcal{Y} \to  S$ are proper surjective holomorphic maps whose fibers are denoted by $X_t:=(\pi_1)^{-1}(t)$ and $Y_t:=(\pi_2)^{-1}(t)$, respectively.

Suppose that there is a bimeromorphic morphism (not just a bimeromorphic map) $f: \mathcal{X} \to \mathcal{Y}$ over $S$.
For some $t\in S$, let $D_t$ be an irreducible component of $Y_t$ such that $D_t$ is of codimension $1$ in $\mathcal{Y}$. Then there exists an irreducible component $C_t$ (equipped with the reduced structure) of $X_t$ that is bimeromorphic to $D_t$, induced by $f$.  In particular, when both $\pi_1$ and $\pi_2$ are families with irreducible fibers, 
$f$ is fiberwise bimeromorphic in the sense of Definition \ref{fiberbimedef}.
\end{thm}

For addressing the issue of indeterminacy, Lemma \ref{elimination of inde}, which deals with the elimination of indeterminacy, serves as a broadly applicable tool to desingularize a bimeromorphic map into a bimeromorphic morphism. However, during the desingularization process, exceptional divisors arise, and some of these exceptional divisors may correspond to the $C_t$ described in Theorem \ref{intro-fiberbime-irredu-singular}. This makes it impossible to guarantee that some component of the original fiber is bimeromorphic to the $D_t$ in Theorem \ref{intro-fiberbime-irredu-singular}.	
	
\textit{How can we overcome this difficulty?} Our {first strategy} is to leverage desingularization theory to analyze the nature of these exceptional divisors during the process of eliminating indeterminacy. In certain cases, it is possible to control the compact exceptional divisors to ensure that they remain uniruled  at each step of the blow-ups used to resolve the indeterminacy. As a result, we obtain a complex analytic analogue of Artin--Matsusaka--Mumford's specialization {\cite[Theorem 1.1]{mm64}} (= Theorem \ref{amm} here).    
\begin{thm}[{=Corollary \ref{smoothcase uniruled}}]\label{smoothcase uniruled-intro}
With Notation \ref{notation 1}, assume that the total space $\mathcal{X}$ is smooth and let $D$ be an irreducible component of $(Y_s)_{\textrm{red}}$ for some $s \in S$ such that $D$ is not uniruled. Then there exists an irreducible component $C$ of $(X_s)_{\textrm{red}}$ such that $C$ is bimeromorphic to $D$, induced by $f$. In particular, if each fiber of $\pi_1$ and $\pi_2$ is irreducible and $Y_t$ is not uniruled for each $t \in S$, then $f$ is fiberwise bimeromorphic in the sense of Definition \ref{fiberbimedef}.
\end{thm}
	
Note that in Theorem \ref{smoothcase uniruled-intro}, the total space of $\pi_1$ is assumed to be smooth.  However, in many situations one must deal with families whose total spaces are singular, a phenomenon that occurs frequently in birational geometry. In such a setting, it is generally difficult to control irreducible exceptional divisors arising from blow-ups so as to ensure that they are uniruled (e.g.,  \cite[Proposition 3.3]{HK15}).

To address such situations, our another strategy is to utilize the deformation behavior of plurigenera (or even the $\ell$-genus for any fixed $\ell \in \mathbb{N}^+$) of families to rule out the possibility that some irreducible exceptional divisor (arising from the elimination of indeterminacy) is ``the $C_t$ presented in Theorem \ref{intro-fiberbime-irredu-singular}, which is bimeromorphic to $D_t$ in Theorem \ref{intro-fiberbime-irredu-singular}".

\begin{thm}[{= Theorem \ref{fiberbime}}]\label{intro-fiberbime}
With Notation \ref{notation 1}, the following conditions hold for some base point $0 \in S$:
\begin{enumerate}[\rm{(}1\rm{)}]
\item\label{intro-fiberbime1}
$\kappa(X_0) \geq 0$;
\item\label{intro-fiberbime2}
Lower semi-continuity: 
For any semi-stable reduction $\mathcal{Z} \to \mathcal{X}$ 
over an open neighborhood $\Delta$ of $0$ whose local model is the unit disk in $\mathbb{C}$, followed by some blowups, such that the fiber $Z_t$ of $\mathcal{Z} \to \Delta$ over $t \in \Delta$ is a proper modification of $X_t$ for $t$ $(\neq 0)$ near $0$, any smooth divisor (not necessarily connected) $Z^{\prime}$ contained in $Z_0$ satisfies $P_m(Z^{\prime}) \le P_m(Z_t)$ for any $m \in \mathbb{N}^+$ and $t$ near $0$;
\item\label{intro-fiberbime4}
Upper semi-continuity: There exists an irreducible component $D$ of $(Y_0)_{\textrm{red}}$ with $P_m(D) \ge P_m(Y_t)$ for $t$ $(\neq 0)$ near $0$ and any $m \in \mathbb{N}^+$.
\end{enumerate}
Then we have the results:
\begin{enumerate}[\rm{(}i\rm{)}]
			\item \label{intro-mainresult1}
			There exists an  irreducible component $C$ of $(X_0)_{\textrm{red}}$ such that $\kappa(C) \geq 0$ and $C$ is bimeromorphic to $D$, induced by $f$;
			\item \label{intro-mainresult2}
			For any other irreducible component $C^{\prime}$ of $(X_0)_{\textrm{red}}$, $\kappa(C^{\prime}) = -\infty$. In particular,  $C$ is the unique irreducible component that is bimeromorphic to $D$;
			\item \label{intro-mainresult3}
			Any other irreducible component $D^{\prime}$ of $(Y_0)_{\textrm{red}}$ cannot satisfy the condition \eqref{fiberbime4} for $D$.
\end{enumerate}
In particular, if each fiber of $\pi_1$ and $\pi_2$ is (globally) irreducible and each point of $S$ satisfies the conditions identical to \eqref{intro-fiberbime1} \eqref{intro-fiberbime2} \eqref{fiberbime4} for $0\in S$, then $f$ is fiberwise bimeromorphic in the sense of Definition \ref{fiberbimedef}.
\end{thm}

As a result, thanks to  the lower semi-continuity and invariance of plurigenera by S. Takayama \cite{Tk07}, we obtain the following fiberwise bimeromorphism result for locally Moishezon families. 
\begin{thm}[{=Theorem \ref{fiberbime taka mois}}]\label{fiberbime taka mois-intro} With Notation \ref{notation 1}, assume that $\pi_1$ is locally Moishezon. Furthermore, for some $s \in S$, suppose the Kodaira dimension $\kappa(X_s) \geq 0$ (see Definition \ref{reducible genera}), and let $D_s$ be an irreducible component of $(Y_s)_{\textrm{red}}$ such that for any  
$m\in \mathbb{N}^+$, the $m$-genera satisfy $P_m(D_s) \geq P_m(Y_t)$ for any $t$ near $s$ (e.g., $Y_s$ has only canonical singularities and $D_s:=Y_s$).   Then, one has the results:
\begin{enumerate}[\rm{(}i\rm{)}] 
\item There exists an irreducible component $C_s$ of $(X_s)_{\textrm{red}}$ such that $\kappa(C_s) \geq 0$ and $C_s$ is bimeromorphic to $D_s$, as induced by $f$;
\item For any other irreducible component $C^{\prime}_s$ of $(X_s)_{\textrm{red}}$, $\kappa(C^{\prime}_s) = -\infty$. In particular,  $C_s$ is the unique irreducible component that is bimeromorphic to $D_s$.
\end{enumerate}
In particular, if for any $t \in S$, both $X_t$ and $Y_t$ are irreducible, $\kappa(X_t) \geq 0$, and $(Y_t)_{\textrm{red}}$ satisfies the condition for $D_s$, then   $f$ is fiberwise bimeromorphic in the sense of Definition \ref{fiberbimedef}. 
\end{thm}

Note that  the local Moishezonness of $\pi_1$ in Theorem \ref{fiberbime taka mois-intro}  is crucial because it enables a desirable extendability property of pluricanonical forms in a suitable setting. This property was established by Takayama  \cite[Theorem 3.1]{Tk07} (= Lemma \ref{Taka-m-extension}).  For the specific roles this extendability plays, refer to the proofs of
 Theorems \ref{fiberbime taka mois} and \ref{fiberbime}.

The condition in Theorem \ref{fiberbime taka mois-intro} is optimal in the sense: the assumption $\kappa(X_s) \geq 0$ cannot generally be omitted, as shown by Example \ref{Hirze example} and the counterexample in \cite[Section 6]{T82}. This is further supported by combining Theorem \ref{specialization-base 1 dim-singular family-intro} with \cite[Theorem 1.1]{T16}, where B. Totaro presented examples of flat projective families with  fibers having terminal
singularities in which the general fiber is rational, but the special fiber is not rational.  
	
Based on the validity of the specialization result for rationality in Theorem \ref{introduction-KT} for smooth families within the scheme category, a natural question arises: With Notation \ref{notation 1}, assume both $\pi_1$ and $\pi_2$ are locally Moishezon or even more general (e.g., locally Fujiki)  smooth families.  
Which natural conditions (in the analytic category) would imply that $f$ is fiberwise bimeromorphic, if the condition $\kappa(X_s) \geq 0$ is no longer assumed?
 However, providing a natural sufficient condition for this problem might be very difficult, as shown by \cite[Theorem 1]{HPT18} and Example \ref{Hirze example} from  H.-Y. Lin \cite{Lin25}. In fact,   B. Hassett--A. Pirutka--Y. Tschinkel \cite[Theorem 1]{HPT18} constructed a smooth family of complex  projective $4$-folds ${X} \to B$  such that the set of points $b \in B$ such that $X_b$ is rational is dense (with respect to  the ordinary complex topology) in $B$, however a very general fiber is even not  stably rational.

\subsection{The main result on specialization  of  bimeromorphic types}
Building on our primary motivation for this paper, we now introduce several results on the specialization of bimeromorphic types. With the above results on fiberwise bimeromorphism in hand, and with our main approach for the specialization,   the task reduces to constructing a global bimeromorphic map between the total spaces of the two families involved, under the condition that the fibers (over general points of the base) of these families are bimeromorphically equivalent.
	
However, a significant challenge arises: noting the fundamental fact that the union of a collection of analytic subsets may not itself be an analytic subset, it becomes philosophically implausible to expect that a collection of bimeromorphic maps between fibers of the two families can be glued into a global bimeromorphic map (over the base) between the total spaces.
	
Fortunately, the insightful observations of D. Mumford \footnote{{Below \cite[Proposition 2.6.1]{Ps77}, in page 61, Persson wrote, ``This observation (pointed out to me by Mumford)...".}} and U. Persson \cite[Lemma 3.1.1, Proposition 3.1.2]{Ps77} provide a way forward. By shrinking the base and performing an appropriate base change, it becomes feasible to construct the desired global bimeromorphic map between the total spaces of the (base-changed) two families. Moreover, this idea has been applied in different contexts by K. Timmerscheidt \cite[Theorem 1]{T82} and T. de Fernex--D. Fusi \cite[Theorem 3.1]{dFF13}.
	
By combining our results (Theorem  \ref{fiberbime taka mois} or Corollary \ref{fibe bime canonical singu-kodaira geq 0}) on fiberwise bimeromorphism with the 
 relative Barlet cycle space theory, we establish several results concerning the specialization of bimeromorphic types.  The main result among them is as follows. 	
	
\begin{thm}[{=Special case of Theorem \ref{specialization-base 1 dim-singular family}}]\label{specialization-base 1 dim-singular family-intro}		
{Let  $\pi_1:\mathcal{X}\to S$ and  $\pi_2:\mathcal{Y}\to S$ be two   locally Moishezon families such that  each fiber of $\pi_1$ and $\pi_2$ has only canonical singularities. Assume that  each fiber of $\pi_1$ is    of non-negative Kodaira dimension. 
Set
$$U_b:= \{t\in S:  X_t \text{ is bimeromorphic to   } Y_t \}.$$
Then $U_b$ is either an at most countable union of proper analytic subsets of $S$ or the whole of $S$. }
\end{thm}

Note that both \cite[Theorem 1]{HPT18} and Theorem \ref{introduction-KT}   deal with smooth families.  Rationality of fibers over a dense (with respect to  the ordinary complex topology) subset of the base is assumed in the former, whereas rationality of the generic fiber is assumed in the latter. In the former case, rationality is not preserved under ``specialization"  while in the latter, rationality is  preserved under specialization. Thus, the deformation closedness of a certain property $(\rm P)$ seems to be very subtle: For a morphism $f:X\to S$ to a curve $S$, if fibers over a subset $U\subseteq S$ satisfy property $(\rm P)$, 
then the size (e.g., the property of being Zariski open, dense, uncountable) of $U$ in $S$ may significantly affect whether  property $(\rm P)$ is preserved  or not under ``specialization".

From the argument of Theorem \ref{specialization-base 1 dim-singular family-intro}, one readily obtains  Corollary \ref{cor-strongly bime}, which gives a criterion for the locally strongly bimeromorphic isotriviality (Definition \ref{def-strongly bime}). Corollary \ref{cor-strongly bime} may  provide a different perspective on F. A. Bogomolov--C. B\"ohning--H.-C. Graf von Bothmer \cite[Theorem 1.1]{BBG16}.
	
Note that, for a long time, exploring and revealing the connections between various topics in birational geometry has been an intriguing theme. For example, the arguments of Y.-T. Siu \cite{Si98} \cite{Si02} and S. Takayama \cite{Tk07} demonstrated the  relationship:
\begin{center}
global positivity of a deformation family $\Rightarrow$ deformation behavior of plurigenera.    
\end{center}
	
Via the arguments and results in this paper,  we also reveal the relationships among four classical topics in bimeromorphic geometry: 
\begin{align*}
&\ \text{deformation behavior of plurigenera (or even $1$-genus)}\\
\Longrightarrow   
&\ \text{fiberwise bimeromorphism} 
\\
\Longrightarrow 
&\ \text{specialization of a fixed bimeromorphic structure}\\
\underset{\scalebox{0.5}{\(\begin{array}{c}\text{method} \\ \text{level}\end{array}\)}}{\Longleftrightarrow}&\ \text{specialization  of bimeromorphic types.}
\end{align*}
	
\subsection{Organization of the paper}
The remainder of this paper is organized as follows:  
	
In Section \ref{pre}, we present several foundational results that are frequently used in the proofs of the theorems in this paper.  
	
In Section \ref{s-for morphism}, we demonstrate that a bimeromorphic morphism between families over the same base is necessarily fiberwise bimeromorphic.  
	
Sections \ref{section-fib bime for map kodaira} and \ref{direct proof-map} focus on providing criteria to determine when a bimeromorphic map over the base is (generalized) fiberwise bimeromorphic, particularly for families with reducible fibers. One type of condition involves controlling the exceptional divisors in a sequence of blowups, while another type addresses the deformation behavior of plurigenera for the families under consideration.  
	
In Section \ref{section-moishezon}, we build upon the arguments in Sections \ref{section-fib bime for map kodaira} and \ref{direct proof-map} to show that a bimeromorphic map between two Moishezon families, where each fiber has Kodaira dimension no less than $0$ and has only canonical singularities, must be fiberwise bimeromorphic. This conclusion relies on Takayama’s results on the deformation behavior of plurigenera for such families.  
	
In Section \ref{s defor limit}, we combine the criteria established in Section \ref{section-moishezon} for fiberwise bimeromorphism in locally Moishezon families with 
the relative Barlet cycle space theory. This combination leads to several results on the specialization of bimeromorphic types for locally Moishezon families (which are not necessarily smooth).  
	
\section{Preliminaries: deformation families and bimeromorphic maps}\label{pre}

Since many discussions go beyond the smooth analytic category,
for convenience of the reader, we compile several fundamental concepts and results in deformation theory and bimeromorphic geometry, which will be frequently utilized in the subsequent sections.	
	
\subsection{Morphisms and  families}\label{ss-preli-1}

Definitions of submersion for morphisms between complex analytic spaces (not necessarily smooth) appear to vary across different references.  Here we adopt:

\begin{defn}[{e.g., \cite[Section 2.18]{Fs76}}]\label{def-subm}
    A holomorphic map $f: X \to Y$ between arbitrary complex spaces is called \emph{a submersion at $p \in X$}, if there exists open neighbourhoods $U \subseteq X$ of $p$, $V \subseteq Y$ of $f(p)$ with $f(U) \subset V$, an open subset $Z \subset \mathbb{C}^k$ (in particular, $Z$ is smooth) and   a biholomorphic map $g: U \to Z \times V$ such that $f=\pi_2\circ g$, where $\pi_2$ is the natural projection $Z\times V \to V$. If $f$ is a submersion at every point of $X$, then $f$ is called \emph{a submersion}. 
\end{defn}

\rem Note that Definition \ref{def-subm} differs from the one in \cite[p. 105]{PR94}, where a submersion is defined in terms of the local freeness of  the relative cotangent sheaf. In fact, a submersion in the sense of \cite[p. 105]{PR94}  is  not necessarily flat (e.g., \cite[Example 2.14]{PR94}), whereas  a submersion in the sense of  Definition \ref{def-subm} is automatically flat.

\begin{lemma}[{\cite[3.21]{Fs76}}]\label{jiayibingding}
	Let $\pi:{X}\to  Y$ be a morphism of complex analytic spaces. Then for any $x\in {X}$, the following conditions are equivalent:
	\begin{enumerate}[\rm{(}1\rm{)}]
		\item $\pi$ is a submersion at $x$.
		\item
	$\pi$ is flat at $x$ and $x$ is a smooth point of the  fiber $X_{\pi(x)}$. 
	\end{enumerate} 
   As a result, if we further assume  that $Y$ is  smooth and $\pi$ is  flat, then the singular locus of ${X}$ must be contained in the union of the singular loci of all the singular fibers (as noted in Subsection \ref{ss-1.1}, the singular loci of a fiber could be the entire fiber itself). 
\end{lemma}

\begin{proof}
The equivalence is precisely the statement of  \cite[3.21]{Fs76}. 

When $\pi$ is flat, this equivalence shows that
$$
x \in\left(X_{\pi(x)}\right)_{\mathrm{reg}} \quad \Longleftrightarrow \quad \pi \text { is a submersion at } x.
$$
Consequently, if $Y$ is smooth, for any $x \in\left(X_{\pi(x)}\right)_{\mathrm{reg}}$, we have  that $X$ is smooth near $x\in X$, based on the definition of a submersion at a point. That is to say,
\[
(X_t)_{\mathrm{reg}} \subseteq  X_{\mathrm{reg}} \cap X_t
\] 
for any $t\in Y$.
Then the lemma follows.
\end{proof}

We now recall the following connectedness criterion for the source space of a morphism.

\begin{lemma}[{\cite[Part 1, Lemma 5.7.5]{Stacks}}]\label{connectedness of total space}
Let $\pi: X\to Y$ be a closed surjective holomorphic map with connected fibers from a complex analytic space $X$ to a connected complex analytic space $Y$. Then $X$ is connected. In particular, when $X$ is further assumed to be locally irreducible (i.e., $\mathcal{O}_{X, x}$ is an integral domain  for any $x\in X$),  $X$ is (globally) irreducible.  
\end{lemma}

Consequently,  the total space of a family in the present paper remains irreducible,  even when the base is shrunk. 
It then follows from  the criterion  for flatness of a proper  morphism over a smooth curve that a family in the present paper is flat.
	
\begin{lemma}[{\cite[Paragraph 1 of Introduction]{Hi75}, \cite[Lemma 2.1]{WZ23}}]\label{flatovercurve}
If the morphism $\pi:\mathcal{X}\to S$ is a family, then $\pi$ is flat. As a result,  the singular locus of the total space of a family must lie in the union of the singular loci of all the singular fibers, based on Lemma \ref{jiayibingding}.  
\end{lemma}

We now introduce the notion of a smooth family, a setting that is commonly studied.

\begin{defn}\label{def-smooth}
A morphism  $f:X\to Y$  is called \emph{smooth} if it is a  submersion.   In particular, based on Lemmata 
\ref{jiayibingding} and   \ref{flatovercurve}, 
a family $\pi:\mathcal{X}\to S$ is  \emph{smooth} if and only if  each fiber of $\pi$ is smooth.    
\end{defn}

 \rem  Note that, by a direct computation of the rank function of the relative cotangent sheaf via \cite[Proposition 1.6]{PR94}, one can apply Lemma \ref{jiayibingding} to conclude that if $X$ is reduced and both $X$ and $Y$ are pure dimensional, then  a morphism $f:X\to Y$ is smooth in the sense of Definition \ref{def-smooth} if and only if it is smooth in the sense of \cite[p. 114]{PR94}, that is, $f$ is flat and the relative cotangent sheaf of $f$ is locally free of rank $\dim X-\dim Y$. 

Recall that, the base curve of a family is  assumed to be smooth in the present paper. It then follows from Lemma
  \ref{flatovercurve}  that the total space of a smooth family is smooth.

In different literature, the definition of a projective morphism seems to be somewhat confusing. We adopt: 
\begin{defn}[{\cite[p. 250]{P94a}}]\label{def-proj}
		A proper surjective holomorphic map $f:X\to Y$ is said to be \emph{projective} if for any relatively compact open set $U \subset Y$, there is an embedding $\iota: f^{-1}(U) \hookrightarrow \mathbb{P}^n \times U$ such that $f = \pi_2 \circ \iota$ with $\pi_2$ being the projection $\mathbb{P}^n \times U \to U$.    
\end{defn}

{Clearly, projectivity is stable under base change, and  the composition of two projective morphisms is also projective. Moreover, a blow-up is  a projective morphism (e.g., \cite[Section 2.7]{P94}). In particular, 
   a composition of finitely many  blow-ups is projective.  Note that the total space of a projective family carries certain positivity information. Consequently,
   one can apply the Ohsawa--Takegoshi extension to extend the pluricanonical forms and achieve lower semi-continuity of plurigenera in certain settings.}

We now introduce the notion of a locally Moishezon morphism, which will serve as the central object of study in this paper.
    
\begin{defn}[{\cite[Definition 3.5]{CP94}}]\label{def-Moish}
A family $p:X\to S$ is said to be \emph{Moishezon} \footnote{Loosely speaking,  the proper morphism in Theorems \ref{introduction-KT} or \ref{intro-NO} corresponds to a Moishezon morphism in the analytic setting, based on the Chow lemma.} if it is bimeromorphically equivalent over $S$ to a surjective projective morphism $q:Y \to S$, i.e., there exists a bimeromorphic (Definition \ref{def-bimeromorphic}) map $g:X \dashrightarrow Y$ such that $p = q \circ g$.  It is called \emph{locally Moishezon} if for any $s\in S$, there exists an open subset $U_s$ such that each $g^{-1}(U_s) \to U_s$ is Moishezon.  The notions of a \emph{Moishezon morphism} and a \emph{locally Moishezon morphism} are also defined similarly.
\end{defn}
	
\subsection{Bimeromorphic geometry}\label{ss prelima}
In this section, we introduce preliminaries on bimeromorphic geometry that will be frequently used. These include proper modifications and meromorphic maps due to Remmert as shown in \cite[$\S$ 2]{Ue75}.
	
\begin{defn}[{\cite[Definition 2.1]{Ue75}}]\label{def-modification}
Let $\pi: \tilde{X} \to X$ be a morphism between two equidimensional, irreducible, and reduced complex analytic spaces. We call $\pi$ a \emph{proper modification} if it satisfies the  conditions:
	\begin{enumerate}[\rm{(}1\rm{)}]
			\item $\pi$ is proper and surjective;
			\item There exist thin (nowhere dense) analytic subsets $\tilde{E} \subseteq \tilde{X}$ and $E \subseteq X$ such that the map
			$$ \pi: \tilde{X} \setminus \tilde{E} \to X \setminus E $$
			is a biholomorphism. Here, $E$ is called the \emph{center} of $\pi$, and $\tilde{E} := \pi^{-1}(E)$ is termed the \emph{exceptional set} of $\pi$.
		\end{enumerate}
When $\tilde{X}$ and $X$ are compact, a proper modification $\pi: \tilde{X} \to X$ is often simply referred to as a \emph{modification}.
\end{defn}

Note that the blow-up of a reduced complex space is again reduced. Hence any finite composition of blow-ups of a reduced and irreducible complex space is a proper modification in the sense of Definition \ref{def-modification}.

\begin{defn}[{\cite[Definition 2.2]{Ue75}}]\label{def-meromorphic}
Let $X$ and $Y$ be two irreducible and reduced complex analytic spaces. A map $\varphi$ from $X$ to the power set of $Y$ is called a \emph{meromorphic map} (denoted by $\varphi: X \dashrightarrow Y$) if the following conditions hold:
\begin{enumerate}[\rm{(}1\rm{)}]
			\item The graph of $\varphi$, denoted by
			$$ \mathcal{G}(\varphi) := \{(x, y) \in X \times Y : y \in \varphi(x)\}, $$
			is an irreducible analytic subset of $X \times Y$;
			\item The projection map $p_X: \mathcal{G}(\varphi) \to X$ is a proper modification, where $\mathcal{G}(\varphi)$ is equipped with the reduced structure.
\end{enumerate}
For any subset $A \subseteq X$ and $B \subseteq Y$, we define the images and preimages of $A$ and $B$ under $\varphi$ as follows:
$$ \varphi(A) := p_Y(p_X^{-1}(A)), \quad \varphi^{-1}(B) := p_X(p_Y^{-1}(B)). $$ 
Furthermore, we call the smallest nowhere dense analytic subset $S_{\varphi}$ of $X$ such that $\varphi$ induces a morphism of $X \setminus S_{\varphi}$ into $Y$ as the \emph{indeterminacy} of $\varphi$.
\end{defn}
Since $p_X$ is proper and $p_Y: \mathcal{G}(\varphi) \to  Y$ is continuous, one can easily see that a meromorphic map sends any compact subset of $X$ (via $p_Y\circ (p_X)^{-1}$) to a compact subset of $Y$. 
\begin{defn}\label{def-bimeromorphic}
A meromorphic map $\varphi: X \dashrightarrow Y$ of complex analytic spaces is called a \emph{bimeromorphic map} if $p_Y : \mathcal{G}(\varphi) \rightarrow Y$ is also a proper modification.  Clearly, a bimeromorphic map is surjective, in the sense that $X$ is onto $Y$ via $p_Y\circ (p_X)^{-1}$.
		
Two complex analytic spaces $X$ and $Y$ are called \emph{bimeromorphically equivalent (or bimeromorphic)} if there exists a bimeromorphic map $\varphi: X \dashrightarrow Y$.
\end{defn}
	
Note that if $\varphi: X \dashrightarrow Y$ is a bimeromorphic map, then the analytic set
$$\{(y, x) \in Y \times X:(x, y) \in \mathcal{G}(\varphi)\} \subseteq Y \times X$$
defines a  bimeromorphic (inverse)  map $\varphi^{-1}: Y \dashrightarrow X$. 
{Since a meromorphic map $Y \dashrightarrow X$ sends any compact subset of $Y$ to a compact subset of $X$, one obtains the well-known fact that
a bimeromorphic morphism is a proper modification.  We will  use
the terms ``bimeromorphic morphism" and 
``proper modification" interchangeably.}

The elimination of indeterminacy, also known as the trivialization of coherent ideal sheaves, is essentially rooted in Hironaka's singularity resolution theory (\cite[pp. 113, 140, 143]{Hi64}). We also refer the reader to other sources such as \cite[Theorem 1.9]{P94}, \cite[Theorem 2.13]{Ue75}, or \cite[Theorem 2.1.24]{MM07} for precise statements, or \cite[Chapter II, Example 7.17.3]{Ha77} for a clearer understanding of this process.

\begin{lemma}[{Elimination of indeterminacy, \cite{Hi64}}]\label{elimination of inde}
Let $X, Y$ be irreducible and reduced complex analytic spaces, and $f: X\dashrightarrow Y$ a bimeromorphic map. Then there exists a proper modification $\varphi: \tilde{X} \to X$, obtained as a composition of a locally finite succession of blow-ups with smooth centers, and a holomorphic map $\tilde{f}: \tilde{X} \to Y$ such that $f \circ \varphi = \tilde{f}$.
\end{lemma}

An easy consequence of the well-known Zariski's main theorem in the analytic setting is:
\begin{lemma}\label{fiber connected after modification}
Let $\pi: \mathcal{X} \to S$ be a family, and $\mu: \tilde{\mathcal{X}} \to \mathcal{X}$ a proper modification that induces proper modifications on the reductions of the general fibers. Then each fiber of $\pi \circ \mu: \tilde{\mathcal{X}} \to S$ is also connected. In particular, when the total space $\tilde{\mathcal{X}}$ is smooth, a general fiber of $\pi \circ \mu$ is a connected complex manifold.
\end{lemma}
\begin{proof}
Since the general fiber of a family is assumed to be irreducible in this paper, and since $\mu$ induces proper modifications on the reductions of the general fibers, the general fiber of $\pi \circ \mu$ is irreducible and thus connected.  Then, each fiber of $\pi \circ \mu$ is connected by Zariski's main theorem \cite[Corollary 1.12]{Ue75}. Consequently, according to the generic smoothness theorem, when $\tilde{\mathcal{X}}$ is smooth, the general fiber of $\pi \circ \mu$ is a connected complex manifold.
\end{proof}
	
In the context of bimeromorphisms, as defined in Definitions~\ref{def-bimeromorphic} and~\ref{def-meromorphic}, it is common to consider the reduction of a complex analytic space.
There is a well-known functoriality property of the reduction map, which is often implicitly used. In the sequel, we will no longer cite this result when using it.
\begin{lemma}[{e.g., \cite[P. 88]{GR84}}]\label{reduction map}
Let $X \to Y$ be a holomorphic map between two complex analytic spaces. Then there exists a holomorphic map $f_{\textrm{red}}: X_{\textrm{red}} \to Y_{\textrm{red}}$ of the reductions such that the diagram
\begin{center}
			\begin{tikzpicture}
				% 绘制交换节点
				\node (X) at (0,1.3) {$X$};
				\node (Y) at (2,1.3) {$Y$};
				\node (Xred) at (0,0) {$X_{\textrm{red}}$};
				\node (Yred) at (2,0) {$Y_{\textrm{red}}$};
				
				% 绘制箭头
				\draw[->] (X) -- (Y)  node[midway, above]{$f$};
				\draw[->] (Xred) -- (X) node[midway, left=1pt]{red};
				\draw[->] (Xred) -- (Yred) node[midway, below]{$f_{\textrm{red}}$};
				\draw[->] (Yred) -- (Y) node[midway, right=1pt]{red};
			\end{tikzpicture}
\end{center}
commutes (both vertical arrows represent the reduction maps of $X$ and $Y$, respectively). 
In particular, all the topological properties of $f_{\textrm{red}}$ are the same as those of $f$, due to the fact that the reduction map is the identity map on the underlying topological spaces.
\end{lemma}
Recently, Koll{\'a}r defined formally and studied a fiberwise bimeromorphism in the complex analytic category, possibly initiated in \cite{KMo92}.
\begin{defn}[{\cite[Definition 26]{Kol22}}] \label{fiberbimedef}
Let $\pi_i: X_i \to S$ ($i=1,2$) be two families. A bimeromorphic map   $\phi: X_1 \dashrightarrow X_2$ over $S$ ({i.e., $\pi_1 = \pi_2 \circ \phi$}) is \emph{fiberwise bimeromorphic} if $\phi$ induces a bimeromorphic map  $$\phi_s: (X_{1,s})_{\textrm{red}} \dashrightarrow (X_{2,s})_{\textrm{red}}$$
for every $s \in S$, where $(X_{i,s})_{\textrm{red}}$ is the reduction of the fiber $X_{i,s} := (\pi_i)^{-1}(s)$. {In particular, both $X_{1,s}$ and $X_{2,s}$ are irreducible as specified in Definitions \ref{def-meromorphic} and  \ref{def-bimeromorphic}.} At times, we may also refer to this as \emph{$X_1$ being fiberwise bimeromorphic to $X_2$ (over $S$ via $\phi$)}.
\end{defn}

Clearly, if a bimeromorphic map $f:X\dashrightarrow Y$ over $S$ is fiberwise bimeromorphic (in the sense of Definition \ref{fiberbimedef}), then  $(X_t)_{\textrm{red}}$ is bimeromorphic to $(Y_t)_{\textrm{red}}$ for any $t\in S$. The converse, however, may be false. H.-Y. Lin \cite{Lin25}  provided a counterexample which also demonstrates that 
even for two smooth families over a same base, a bimeromorphic map between them  over the base may not be  fiberwise bimeromorphic.

\begin{ex}[{\cite{Lin25}}]\label{Hirze example}%\label{notion fiberwise}
Consider the bimeromorphic map 
 (over $\mathbb{P}^1$) between $f: \mathbb{P}^1 \times \mathbb{P}^1\times Z \to \mathbb{P}^1$ and  $g: \Sigma_n\times Z \to \mathbb{P}^1$, which is  induced by the natural bimeromorphic map (over $\mathbb{P}^1$) between $\mathbb{P}^1 \times \mathbb{P}^1$ and $\Sigma_n$, where $Z$ is any compact manifold and $\Sigma_n$ is a Hirzebruch surface.  This map is not fiberwise bimeromorphic, even though  $f^{-1}(t)$ is bimeromorphic to $g^{-1}(t)$ for any $t \in \mathbb{P}^1$.  
\end{ex}
	
Let $f: X \dashrightarrow Y$ and $g: Y \dashrightarrow Z$ be two meromorphic maps. As is well-known, the composition $g \circ f$ does not necessarily exist as a meromorphic map. However, if $f$ is surjective, then we can define the composition as shown in \cite[pp. 16, 17]{Ue75}, such that $g \circ f$ is also a meromorphic map. Noting that $\mathcal{G}_{f}$ is equal to the closure of $\mathcal{G}_{f|_{X \setminus S_f}}$ in $X \times Y$, we can obtain the following result by elementary analysis.
	
\begin{lemma}\label{composition}
Let $X$, $Y$, and $Z$ be reduced and irreducible complex analytic spaces, $f: X \dashrightarrow Y$ a bimeromorphic map, and $g: Y \to Z$ a holomorphic map. Then we have $(g \circ f)(x) = g(f(x))$, where the composition $g \circ f$ is defined as shown in \cite[pp. 16, 17]{Ue75}.
\end{lemma}
	
As a direct application of Lemma \ref{composition}, we have: 
\begin{lemma}\label{preserve fiber}
Let 
$$
\pi_1: \mathcal{X} \to S \quad \text{and} \quad \pi_2: \mathcal{Y} \to S
$$
be two families. Suppose that there is a bimeromorphic map
$f: \mathcal{X}\dashrightarrow \mathcal{Y}$ over $S$. Then for any $s \in S$, $f$ exactly sends $X_s$ to $Y_s$, i.e., $f(X_s) = Y_s$ set-theoretically.
\end{lemma}

Based on Lemmata \ref{connectedness of total space} and \ref{preserve fiber},  we obtain: 
\begin{lemma}\label{restriction of bime}
Let
$$
\pi_1: \mathcal{X} \to S \quad \text{and} \quad \pi_2: \mathcal{Y} \to S
$$
be two families. Suppose that there is a bimeromorphic map $f: \mathcal{X}\dashrightarrow \mathcal{Y}$ over $S$. Then $f: \mathcal{X}\dashrightarrow \mathcal{Y}$ is a bimeromorphic map over any connected open subset $S_1 \subseteq S$.
\end{lemma}
	
In our argument, the \emph{local version of the Chow lemma} plays an important role.
\begin{lemma}[{\cite[Corollary 2.9]{P94}}]\label{local version of chow}
Let $f: X \to Y$ be a proper modification of reduced and irreducible complex spaces. Let $V \Subset Y$ be an open relatively compact set. Then there exist blow-ups $g: U \to V$ with center $D \subseteq V$ and $h: U \to f^{-1}(V)$ with center $f^{-1}(D)$, the analytic preimage of $D$, such that $g = f \circ h$.
\end{lemma}
	
By elementary analysis of thin analytic subsets in the total spaces of both families, we can readily demonstrate that a bimeromorphic map between two families over the same base must generally be fiberwise bimeromorphic in the following sense. Note that, in this paper, we assume that the total space of a family is reduced (and even normal),  and so the general fiber is also reduced as established in \cite[Lemma 1.4]{Fu78/79}. Also, we assume that the general fiber of each family in this paper is irreducible. We are thus able to discuss the bimeromorphism between \(X_t\) and \(Y_t\) for a general \(t\), as defined in Definitions \ref{def-modification} and \ref{def-bimeromorphic}.
	
\begin{lemma}\label{obs 1}
With Notation \ref{notation 1}, there exists a proper analytic subset $\Gamma \subseteq S$ such that for any $t \in S \setminus \Gamma$, the fiber $X_t$ is bimeromorphic to $Y_t$. 
\end{lemma}
\begin{proof}
Here we do not employ elementary methods but instead opt for alternative approaches, which can be viewed as a warm-up for a proof of Theorem \ref{fiberbime} in some sense.
		
Note that a subset $A$ of an irreducible curve $C$ is a proper analytic subset of $C$ if and only if $A$ consists of isolated points of $C$. Therefore, it suffices to prove this lemma locally. By Lemma \ref{restriction of bime},  $f: \mathcal{X} \dashrightarrow \mathcal{Y}$ is a bimeromorphic map such that $\pi_1 = \pi_2 \circ f$ over any connected open subset $S_1 \subseteq S$. We always use $S$ to denote the base, even after it is shrunk to a smaller connected open subset. Denote by $X_s$ and $Y_s$ the special fibers of $\pi_1,\pi_2$ at some $s\in S$, respectively.  Now we perform a series of operations, as shown in the commutative diagram
$$
		\xymatrix{
			%Y\ar[rr] \ar@{^{(}->}[ddr] & & \tilde Y\ar@{^{(}->}[ddr]|-{} 
			& &\mathcal{X}_2\ar[dr]^{g}\ar[dl]_{\mu}
			\\
			&\mathcal{X}_1\ar@{-->}[rr]^{f_1} \ar@{->}[d] & &\mathcal{Y}_1\ar@{->}[d] \\
			&\mathcal{X} \ar@{-->}[rr]^{f} \ar[dr]_{\pi_1}
			&  &    {\mathcal{Y}} \ar[dl]^{\pi_2}   \\
			& & S                }.
$$
		
More precisely, by the irreducibility of $\mathcal{Y}$, we can take a desingularization $\mathcal{Y}_1 \to \mathcal{Y}$ of the normal complex space $\mathcal{Y}$ such that $\mathcal{Y}_1$ is a connected complex manifold (e.g., \cite[Theorem 5.4.2]{AHV18} or \cite[Theorem 2.12]{Ue75}). Since the fibers of $\mathcal{Y} \to \Delta$ are of codimension 1, we can assume that this desingularization induces a proper modification on the reduction of each fiber $Y_t$ over $t \neq s$, after shrinking $S$.
		
Similarly, let $\mathcal{X}_1 \to \mathcal{X}$ be the desingularization of $\mathcal{X}$ such that this desingularization induces a proper modification on the reduction of each fiber $X_t$ over $t \neq s$, after shrinking $S$. Now we get a bimeromorphic map $f_1: \mathcal{X}_1 \dashrightarrow \mathcal{Y}_1$.
		
By Lemma \ref{elimination of inde}, we eliminate the indeterminacy of $f_1$ by $\mu$, which is the composition of a locally finite sequence of blow-ups with smooth centers. As a result, we obtain a bimeromorphic morphism $g: \mathcal{X}_2 \to \mathcal{Y}_1$. Note that $\mu$ induces a proper modification on the general fiber of $\mathcal{X}_1 \to S$, by the easy fact that vertical $\mu$-exceptional divisors only locate in isolated fibers.
		
Since both $\mathcal{X}_2$ and $\mathcal{Y}_1$ are smooth, the center of  the proper modification $g$ is of codimension $\geq 2$ by \cite[p. 215]{GR84}. Consequently, $g$ also induces a proper modification on the general fiber of $\mathcal{Y}_1 \to S$ (using again the fact that vertical $g$-exceptional divisors only locate in isolated fibers), implying that $X_t$ is bimeromorphic to $Y_t$ for general $t \in S$.
\end{proof}	
	
\section{Fiberwise bimeromorphism of a bimeromorphic morphism}\label{s-for morphism}
Now we present an observation on fiberwise bimeromorphism for a bimeromorphic morphism: A bimeromorphic morphism (not just a bimeromorphic map) between two families over the same base must be fiberwise bimeromorphic. This observation demonstrates, in some sense, that the ramification locus does not pose an obstacle to a bimeromorphic morphism from being fiberwise bimeromorphic. The argument of Theorem \ref{new-fiberbime-irredu-singular} can be regarded as an elementary analysis of the ramification divisor and its image (when the morphisms involved are assumed to be smooth). Note that all the results in this section do not generally hold for a bimeromorphic map between two families over the same base. We will investigate the question of fiberwise bimeromorphism for a general bimeromorphic map in subsequent sections, employing different methods.
	
\begin{thm}\label{new-fiberbime-irredu-singular}	
Let $\mathcal{X}$,  $\mathcal{Y}$ and $S$ be complex analytic spaces. Assume that $\mathcal{X}$ is reduced (not necessarily normal) and irreducible, $\mathcal{Y}$ is normal, and $S$ is irreducible (not necessarily $1$-dimensional).	
Assume further that both $\pi_1: \mathcal{X} \to S$ and $\pi_2: \mathcal{Y} \to  S$ are proper surjective holomorphic maps whose fibers are denoted by $X_t:=(\pi_1)^{-1}(t)$ and $Y_t:=(\pi_2)^{-1}(t)$, respectively.

Suppose that there is a bimeromorphic morphism (not just a bimeromorphic map) $f: \mathcal{X} \to \mathcal{Y}$ over $S$.
For some $t\in S$, if $D_t$ is an irreducible component of $Y_t$ that   is of codimension $1$ in $\mathcal{Y}$, 
then there exists an irreducible component $C_t$ (equipped with the reduced structure) of $X_t$ that is bimeromorphic to $D_t$ (also equipped with the reduced structure), induced by $f$. 
\end{thm}
\begin{proof}
This proof draws inspiration from that of \cite[Chapter II, Section 1.d, 1.12 Lemma]{Ny04}.
		
According to \cite[(1.17) + Theorem 1.19]{PR94}, there exists an analytic subset $V$ of $\mathcal{Y}$ with codimension at least $2$, such that $\operatorname{dim}_x X_y = 0$ for any $x \in X_y$ and any $y \in \mathcal{Y} \setminus V$. 
It follows from the connectedness lemma (e.g., \cite[Corollary 1.12]{Ue75}) for 
the proper modification $f$ that 
 each fiber of $f$ is connected. Consequently, $f|_{f^{-1}(\mathcal{Y} \setminus V)}: f^{-1}(\mathcal{Y} \setminus V) \to \mathcal{Y} \setminus V$ is injective (bijective) and thus biholomorphic, as established by \cite[p. 166]{GR84}.
		
Since $f|_{X_t}$ is proper, it sends any analytic subset of $X_t$ to an analytic subset of $Y_t$ based on Remmert's proper mapping theorem. Additionally, $f$ is surjective by the definition of a bimeromorphic morphism. Consequently,  there exists an irreducible component $C_t$ of $X_t$ such that $f(C_t) = D_t$ by the irreducibility of $D_t$.
		
In view of the codimensions of $V$ and $D_t$, it follows that $D_t\nsubseteq V$, and consequently, $C_t \nsubseteq f^{-1}(V)$. Clearly, $D_t \cap V$ is a thin analytic subset of $D_t$, and $C_t \cap f^{-1}(V)$ is a thin analytic subset of $C_t$. Hence, one can easily check by definition that $f: C_t \to D_t$ is bimeromorphic. This completes the proof.
\end{proof}

As an immediate corollary of Theorem \ref{new-fiberbime-irredu-singular}, we can apply Lemmata \ref{flatovercurve} and \ref{connectedness of total space} to obtain the following result in the setting of families, based on the dimension formula (e.g., \cite[Proposition 2.11]{PR94}) for the flat morphism.

\begin{cor}\label{morphism-cor-25}
Consider two families
$$
\pi_1: \mathcal{X} \to S \text{ and } \pi_2:  \mathcal{Y} \to S.
$$
Denote by $X_t:=(\pi_1)^{-1}(t)$ and $Y_t:=(\pi_2)^{-1}(t)$ their fibers on $t\in S$, respectively. 
Suppose that there is a bimeromorphic morphism $f: X \rightarrow Y$ over $S$ such that $\pi_1=\pi_2 \circ f$. Then, for any $t \in S$ and any irreducible component $D_t$ of $\left(Y_t\right)_{\mathrm{red}}$, there exists an irreducible component $C_t$ of $\left(X_t\right)_{\mathrm{red}}$ that is bimeromorphic to $D_t$, induced by $f$. In particular, if we further assume that each fiber of $\pi_1$ and $\pi_2$ is irreducible, then $\left(X_t\right)_{\mathrm{red}}$ is bimeromorphic to $\left(Y_t\right)_{\mathrm{red}}$, induced by $f$, i.e., $f$ is fiberwise bimeromorphic in the sense of Definition \ref{fiberbimedef}.
\end{cor}

\rem Corollary \ref{morphism-cor-25} is used in \cite{HLR26} many times, for instance to obtain the fiberwise bimeromorphisms of the extended relative bimeromorphic contractions in \cite[Proposition 4.24, Theorem 6.1]{HLR26};   Note that Fujiki’s openness of bimeromorphism \cite[Lemma 15]{Fu84}, together with Corollary \ref{morphism-cor-25}, yields a result of the following form: under suitable natural assumptions, for a morphism (which is not assumed a priori to be bimeromorphic) $f: \mathcal{X}\to \mathcal{Y}$ over $S$, if a induced morphism  $f_s: X_s \to Y_s$ is bimeromorphic for some $s\in S$, then  $f$ is fiberwise bimeromorphic.

As J. Koll{\'a}r \cite{Ko25} pointed out,
 an argument similar to Theorem \ref{new-fiberbime-irredu-singular} applies in the absolute setting. The only adjustment is that, in Paragraph 3 of its proof, we choose an irreducible component $C$ of $f^{-1}(D)$ with $f(C) = D$ to be the desired $C$.

\begin{cor}[{\cite{Ko25}}]\label{Kol-abso}
Let $\mathcal{X}$ and  $\mathcal{Y}$  be complex analytic spaces such that  $\mathcal{X}$ is reduced  and irreducible, $\mathcal{Y}$ is normal. Let $D \subseteq \mathcal{Y}$ be an irreducible analytic subset of  codimension $1$. 
Assume that  we have  a bimeromorphic morphism  $f: \mathcal{X} \to \mathcal{Y}$.  Then there exists an irreducible analytic subset $C$ of $\mathcal{X} $ that is bimeromorphic to $D$, induced by $f$.
\end{cor}
        
\section{Fiberwise bimeromorphism of a bimeromorphic map}\label{section-fib bime for map kodaira}
In this section, we provide criteria to determine when a bimeromorphic map between two families over the same base is fiberwise bimeromorphic. It is worth noting that an important class of families, studied in \cite{Tk07}, satisfies this criterion. 
In Subsection \ref{4.1-0726}, motivated by desingularization theory, we investigate this topic in the setting of families with smooth total spaces. In Subsection \ref{4.2-0726}, we approach the problem through the deformation behavior of plurigenera.
	
The key idea is that we utilize the deformation behavior of the plurigenera of fibers to arrange the special fiber in such a way that it avoids being covered by any ``bad locus" of the bimeromorphic map involved. More precisely, in Subsection \ref{4.2-0726} and Section \ref{section-moishezon}, we primarily rely on the extension behavior of pluricanonical forms to address the fiberwise bimeromorphism problem.
	
Recall that the \emph{$m$-genus} $P_m(X)$ of a compact complex manifold $X$ is defined by
$$P_m(X):=\dim_{\mathbb{C}} H^0(X,K_X^{\otimes m})$$
with the canonical bundle $K_X$ of $X$ for every positive integer $m$. It is a bimeromorphic invariant and the \emph{$m$-genus of an arbitrary compact complex irreducible and reduced analytic space $M$} can be defined as the $m$-genus of its arbitrary non-singular model. Let $L$ be a holomorphic line bundle over $M$. 
Set the linear system associated to $\mathcal{L}:=H^{0}(M,L)$ as $|\mathcal{L}|=\{\Div(s): s\in \mathcal{L}\}$ and its {base point locus} $\Bl_{|\mathcal{L}|}$.
Now consider $\mathcal{L}_p:=H^{0}(M,L^{\otimes p})$ for a positive integer $p$ and the Kodaira map $\Phi_p:=\Phi_{\mathcal{L}_p}$ associated to $L^{\otimes p}$.
Set
$$
\varrho_p=
	\begin{cases}
		\max\{\rk\ \Phi_p: x\in M\setminus \Bl_{|\mathcal{L}_p|}\}, &\text{if $\mathcal{L}_p\neq\{0\}$},\\
		-\infty, &\text{otherwise}.
	\end{cases}
$$
The \emph{Kodaira--Iitaka dimension} $\kappa(L)$ of $L$ is $\max\{\varrho_p:p\in \mathbb{N}^+\}$. 
A classical result of S. Iitaka says that there exist positive numbers $\alpha,\beta$ and a positive integer $m_0$ such that for any integer $m\geq m_0$, the inequalities hold
\begin{equation}\label{asymp-ineq}
\alpha m^{\kappa(L)}\leq h^0(M, L^{\otimes md})=\dim_{\mathbb{C}}H^0(M, L^{\otimes md})\leq \beta m^{\kappa(L)},
\end{equation}
where $d$ is some positive integer depending on $L$. In particular, when $h^0(M, L)\neq 0$, one can take $d=1$ in \eqref{asymp-ineq}. The \emph{Kodaira dimension} $\kappa(M)$ of $M$ is defined as the Kodaira--Iitaka dimension of the canonical bundle of its arbitrary non-singular model.
	
For a more general compact complex analytic space, we adopt:  
\begin{defn}\label{reducible genera}
Let $X$ be a compact complex analytic space. For every positive integer $m$, the \emph{$m$-genus} $P_m(X)$ of $X$ is defined by
\[
P_m(X) := \sum_{i \in I} P_m({X}_i),
\]
where $\{X_i\}_{i \in I}$ is the set of irreducible components of the reduced space $X_{\textrm{red}}$, while the \emph{Kodaira dimension} $\kappa(X)$ of $X$ is defined by
\[\kappa(X) := \max_{i \in I} \kappa({X}_i).\]
\end{defn}
	
\subsection{From the point of view of desingularization theory}\label{4.1-0726}
Let $X$ be  a   reduced and irreducible compact complex analytic space. 
Recall that
(e.g., \cite[Definition 2.8]{Hr08})  a \emph{rational curve in $X$} is defined as the image (with the reduced structure) of a nonconstant morphism $f: \mathbb{P}^1 \to X$;   $X$ is called  \emph{uniruled}, if through a general point of $X$ there exists a rational curve in $X$.   By the definitions of uniruledness and a proper modification, if $f:Z\to W$ is a proper modification  and $W$ is not uniruled, then $Z$ is also not  uniruled.

Motivated by the basic fact that a blow-up of a smooth manifold along a compact irreducible smooth center has a uniruled exceptional divisor, we now present Theorem \ref{fiberbime global-D kappa geq 0} concerning fiberwise bimeromorphism.
Recall that the weak factorization theorem states that a bimeromorphic map between two compact complex manifolds can be factored through a series of blowups and blowdowns with smooth centers. As a warm-up for Theorem \ref{fiberbime global-D kappa geq 0}, we begin with: 
\begin{observation}
With Notation \ref{notation 1}, assume that the map $f$ can be factored through a series of blowups and blowdowns, where each irreducible component of the exceptional divisors in each step (each blowup and the inverse of each blowdown) is uniruled. 
Assume further that each irreducible component of (the reduction of) each fiber of both $\pi_1$ and $\pi_2$ is not uniruled. Then, for any $s \in S$ and any irreducible component $D_s$ of the reduction of $Y_s$ there exists an irreducible component $B_s$ of the reduction of $X_s$ that is bimeromorphic (via Theorem \ref{new-fiberbime-irredu-singular}) to $D_s$, induced by $f$.
\end{observation}
	
\begin{thm}\label{fiberbime global-D kappa geq 0}	
Let $\mathcal{X}$,  $\mathcal{Y}$ and $S$ be complex analytic spaces. Assume that both $\mathcal{X}$ and $\mathcal{Y}$ are reduced and irreducible and that $S$ is a connected smooth curve. Assume further that both $\pi_1: \mathcal{X} \to  S$ and $\pi_2: \mathcal{Y} \to  S$ are proper surjective holomorphic maps and we have a bimeromorphic map $f: \mathcal{X} \dashrightarrow \mathcal{Y}$ over $S$. 	 
For some $s \in S$, let $D$ be an irreducible component of $(Y_s)_{\textrm{red}}$ such that $D$ is not uniruled. Furthermore, assume the conditions:
\begin{enumerate}[\rm{(}1\rm{)}]
\item\label{20240730-1} There exists a proper modification $\sigma_1: \mathcal{V} \to \mathcal{X}$ obtained as a composition of a locally finite succession of blow-ups, each with the center of codimension no less than two, and each compact irreducible exceptional divisor being uniruled.
\item \label{Y center 2 condition}  There also exists a proper modification $\sigma_2: \mathcal{W} \to \mathcal{Y}$ obtained as a composition of a locally finite succession of blow-ups, with each center being of codimension no less than two. Moreover, assume that $\mathcal{W}$ is normal.
\item $g: \mathcal{V} \to \mathcal{W}$ is a proper modification such that $\sigma_2 \circ g = f \circ \sigma_1$. 
\end{enumerate}
Then, there exists an irreducible component $C$ of $(X_s)_{\textrm{red}}$ such that $C$ is bimeromorphic to $D$, induced by $f$. In particular, if we further assume that each fiber of $\pi_1$ and $\pi_2$ is irreducible and $Y_t$ is not uniruled   for each $t \in S$, then $f$ is fiberwise bimeromorphic in the sense of Definition \ref{fiberbimedef}.
\end{thm}
\begin{proof}
				
Since $\mathcal{Y}$ is reduced and irreducible and  $S$ is a connected smooth curve,  $\pi_2$ is flat by the criterion (e.g., \cite[Paragraph 1 of Introduction]{Hi75}, \cite[Lemma 2.1]{WZ23}) for flatness of a morphism over a smooth curve. 
It follows that $D$ is of codimension $1$ in $\mathcal{Y}$, based on the dimension formula (e.g., \cite[Proposition 2.11]{PR94}) for the flat morphism $\pi_2$.
By the condition \eqref{Y center 2 condition},  the final strict transform $\tilde{D}$  of $D$ under $\sigma_2$ is well-defined such that  $\tilde{D}$  is bimeromorphic to  $D$ (via a proper modification) and  is  of codimension $1$ in $\mathcal{W}$. In particular, $\tilde{D}$ is not uniruled.

 Clearly, both $\mathcal{V}$ and  $\mathcal{W}$  are irreducible.   
Applying Theorem \ref{new-fiberbime-irredu-singular} to $\pi_1\circ \sigma_1$ and $\pi_2\circ \sigma_2$,  one obtains an irreducible component $F$ of the reduction of $V_s$ that is bimeromorphic to $\tilde{D}$. 
 Since $\tilde{D}$ is not uniruled, we can conclude that $F$ cannot be an exceptional divisor of $\sigma_1$, as each compact irreducible exceptional divisor of $\sigma_1$ is uniruled, as required in Condition \eqref{20240730-1}. 
 
 Consequently, $F$ must be the final strict transform of some irreducible component of the reduction of $X_s$ under $\sigma_1$. This completes the proof.
\end{proof}

{The conditions in Theorem \ref{fiberbime global-D kappa geq 0} appear somewhat artificial, yet Corollary \ref{smoothcase uniruled} offers a natural example that illustrates them. The key point for this setting is that one can control the irreducible compact exceptional divisor to be uniruled in each step of the blow-ups for eliminating the indeterminacy (Lemma \ref{elimination of inde}).  
Consequently, if in Theorem \ref{fiberbime global-D kappa geq 0} we take $\sigma_2$ to be the identity on $\mathcal{Y}$ and let $\sigma_1$ be such an elimination of indeterminacy,  we obtain the corollary below, a situation commonly encountered in the literature.}

\begin{cor}\label{smoothcase uniruled}
With Notation \ref{notation 1}, assume that the total space $\mathcal{X}$ is smooth and let $D$ be an irreducible component of $(Y_s)_{\textrm{red}}$ for some $s \in S$ such that $D$ is not uniruled. Then, there exists an irreducible component $C$ of $(X_s)_{\textrm{red}}$ such that $C$ is bimeromorphic to $D$, induced by $f$. In particular, if one assumes further that each fiber of $\pi_1$ and $\pi_2$ is irreducible and $Y_t$ is not uniruled  for each $t \in S$, then $f$ is fiberwise bimeromorphic in the sense of Definition \ref{fiberbimedef}.
\end{cor}

Note that in Corollary \ref{smoothcase uniruled}, the total space of $\pi_1$ is assumed to be smooth. We now address the case where the total spaces involved are singular. This situation often appears in birational geometry; for instance, in many cases, we cannot ensure that the contraction of a manifold is smooth. However, in this case, compact exceptional divisors of blow-ups cannot be easily controlled to be uniruled (e.g., \cite[Proposition 3.3]{HK15}). 
Note that Koll\'ar \cite[p. 289]{Kol96} refers to a proper modification (in the scheme category)  as a \textit{uniruled modification}, if every irreducible component of the exceptional set of this modification is uniruled.  We refer the reader to \cite[Proposition 4]{Abh56}, \cite[p. 8, Proposition 1.3]{KMo92}, \cite[p. 291, Theorem 1.10]{Kol96}, \cite[p. 31, Proposition 1.43]{D01}, \cite[Corollary 1.6]{HM07}, etc., for sufficient conditions for irreducible exceptional divisors to be uniruled (in the scheme category). 

Corollary \ref{smoothcase uniruled} can be viewed as a complex analytic analogue of Artin--Matsusaka--Mumford's specialization theorem:
\begin{thm}[{\cite[Theorem 1.1]{mm64}}]\footnote{Just below \cite[Theorem 1.1]{mm64}, it is mentioned that this theorem was pointed out by M. Artin.}\label{amm}
Let $X$ be a complete abstract variety, $Y$ an abstract
variety and $T$ a birational correspondence between $X$ and $Y$. Let o be a
discrete valuation-ring with the quotient field $k$, such that $X$, $Y$ and $T$ are defined over $k$. Let $(X', Y', T')$ be a specialization of $(X, Y, T)$ over $o$ and assume that $X', Y'$ are abstract varieties and that $X'$ is complete. When $Y'$ is not a ruled variety, there is a component $T''$ of $T'$ with the coefficient $1$ in $T'$ such that $T''$ is a birational correspondence between $X'$ and $Y'$ and that $pr_i(T'- T'') =0$ for $i= 1, 2$.
\end{thm}

In the algebraic setting, for a proper modification $\pi$ with a smooth target, every irreducible component of the exceptional locus of $\pi$ is ruled (e.g., \cite[Proposition 1.43]{D01}). 
This might lead one to expect the following: if $X$ is a complex manifold and $Z$ is a compact connected submanifold of $X$ (not necessarily projective), then the exceptional divisor of a blowup $f:\tilde{X} \to X$ along $Z$ is ruled. 
However, this expectation might not be valid, 
as demonstrated by the following example provided by H.-Y. Lin \cite{Lin25}.

\begin{ex}[{\cite{Lin25}}]\label{exc not uniruled}
 Let $S$ be a $K3$ surface containing no curves.  Consider the projective compactification $X:=\mathbb{P}(T_S\bigoplus \mathcal{O}_S)$ of the total space of $T_S$.  If we blow up 
 $X$  along the zero section of $T_S$, the exceptional divisor is isomorphic to $\mathbb{P}(T_S)$.
 However,  the tangent bundle  $T_S$  is not trivial over any Zariski dense open subset of $S$ by the dimension analysis and the  assumption that $S$ contains no curves.   
\end{ex}
 
In the next subsection, we will reveal the interesting connection between the deformation behavior of plurigenera and fiberwise bimeromorphism: 

deformation behavior of plurigenera (or even $1$-genus) $\Rightarrow$ fiberwise bimeromorphism. \\
More precisely, we mainly utilize the deformation behavior of plurigenera (or $\ell$-genus for some fixed $\ell$) to investigate questions of fiberwise bimeromorphism. In fact, we will provide more interesting information via the deformation behavior of plurigenera. Our Theorem \ref{fiberbime} not only establishes the existence of an irreducible component $C$ that is bimeromorphic to something as $D$ of
Corollary \ref{smoothcase uniruled}
 but also identifies this component $C$.
	
\subsection{From the point of view of the deformation behavior of plurigenera}\label{4.2-0726}

We now present Theorem \ref{fiberbime}, which broadly illustrates the relationship:
\begin{center}
deformation behavior of plurigenera $\Rightarrow$ fiberwise bimeromorphism.
\end{center}
Its proof will be deferred to the next section. 
    
\begin{thm}\label{fiberbime}
With Notation \ref{notation 1}, the following conditions hold for some base point $0 \in S$:
\begin{enumerate}[\rm{(}1\rm{)}]
\item\label{fiberbime1}
$\kappa(X_0) \geq 0$;
\item\label{fiberbime2}
Lower semi-continuity: 
For any semi-stable reduction $\mathcal{Z} \to \mathcal{X}$ 
\footnote{\label{localversion-basechange} We should first desingularize $\mathcal{X}$ to make it smooth and shrink the disk, ensuring that the central fiber is the sole singular fiber (i.e., making the family $\mathcal{X} \to \Delta$ a $1$-parameter degeneration). Note that here we denote the finite base change of $\mathcal{X}$ by $\mathcal{X}$ itself, as stipulated in the introduction. For  precise details on the application of the semi-stable reduction theorem, we refer the reader to  paragraph 2 of the proof of Claim \ref{claim-one element}.}
over an open neighborhood $\Delta$ of $0$ whose local model is the unit disk in $\mathbb{C}$, followed by some blowups, such that the fiber $Z_t$ of $\mathcal{Z} \to \Delta$ over $t \in \Delta$ is a proper modification of $X_t$ for $t$ $(\neq 0)$ near $0$, any smooth divisor (not necessarily connected) $Z^{\prime}$ contained in $Z_0$ satisfies $P_m(Z^{\prime}) \le P_m(Z_t)$ for any $m \in \mathbb{N}^+$ and $t$ near $0$;
\item\label{fiberbime4}
Upper semi-continuity: There exists an irreducible component $D$ of $(Y_0)_{\textrm{red}}$ with $P_m(D) \ge P_m(Y_t)$ for $t$ $(\neq 0)$ near $0$ and any $m \in \mathbb{N}^+$.
\end{enumerate}
Then we have the results:
\begin{enumerate}[\rm{(}i\rm{)}]
			\item \label{mainresult1}
			There exists an  irreducible component $C$ of $(X_0)_{\textrm{red}}$ such that $\kappa(C) \geq 0$ and $C$ is bimeromorphic to $D$, induced by $f$;
			\item \label{mainresult2}
			For any other irreducible component $C^{\prime}$ of $(X_0)_{\textrm{red}}$, $\kappa(C^{\prime}) = -\infty$. In particular,  $C$ is the unique irreducible component that is bimeromorphic to $D$;
			\item \label{mainresult3}
			Any other irreducible component $D^{\prime}$ of $(Y_0)_{\textrm{red}}$ cannot satisfy the condition \eqref{fiberbime4} for $D$.
\end{enumerate}
In particular, if each fiber of $\pi_1$ and $\pi_2$ is (globally) irreducible and each point of $S$ satisfies the conditions identical to \eqref{fiberbime1} \eqref{fiberbime2} \eqref{fiberbime4} for $0\in S$, then $f$ is fiberwise bimeromorphic in the sense of Definition \ref{fiberbimedef}.
\end{thm}

\rem {The formulation of  condition \eqref{fiberbime2} in Theorem \ref{fiberbime} appears  somewhat unfamiliar and artificial.   In fact, this condition is naturally satisfied in a classical setting where all the pluricanonical forms on $Z^{\prime}$ can be extended to nearby fibers. For instance, this occurs in the case of a projective (or even Moishezon) family $\mathcal{Z} \to \Delta$ (\cite[Theorem 3.1]{Tk07}).  
Moreover, it is also expected (\cite[Conjecture 2.1]{Si02a}, \cite[Remark 1.2]{Ct07})  to  hold for  K\"ahler  families.}
For related results whose conditions are more natural, we refer the reader to Theorem \ref{fiberbime taka mois} and Corollary \ref{fibe bime canonical singu-kodaira geq 0}, or Lemma \ref{Taka-m-extension}.  Furthermore, note that the plurigenera in the conditions can be replaced by any other bimeromorphic invariant, which can be easily observed in the proof of Theorem \ref{fiberbime}.

\rem Note that Theorem \ref{fiberbime} provides a criterion for a fundamental question in algebraic geometry: Given two families $\pi_1$ and $\pi_2$ over the same smooth curve, one may ask for determining whether there exists a bimeromorphic map between the two families over the base. To answer it for the nonexistence, if the conditions \eqref{fiberbime2} and \eqref{fiberbime4} hold, and if there exist at least two irreducible components $(X_0)_{\textrm{red}}$ with nonnegative Kodaira dimension, then there does not exist an open connected set $U$ near $0$ such that $\pi_1^{-1}(U)$ is bimeromorphic to $\pi_2^{-1}(U)$ over $U$. {Furthermore, by applying Theorem \ref{fiberbime} to the case where 
$f$ is the identity, one reveals another basic phenomenon: Let $\pi:\mathcal{X}\to S$ be a family with the lower semicontinuity property in Theorem \ref{fiberbime}, and assume that an irreducible component of $(X_0)_{\mathrm{red}}$ satisfies the upper semicontinuity property in Theorem \ref{fiberbime}. Then it is impossible for $(X_0)_{\mathrm{red}}$ to contain two irreducible components both of Kodaira dimension at least \(0\).}

\rem In some sense, Theorem \ref{fiberbime} conveys that the deformation behavior of plurigenera entirely determines the fiberwise bimeromorphic structure of a family, as we fix the (biholomorphic structure of) base space and fix the bimeromorphic structure of the total space. In fact, we can even obtain more, in Subsection \ref{ss refinement}, that the deformation behavior of $\ell$-genus (for any fixed $\ell \in \mathbb{N}^+$) entirely determines the fiberwise bimeromorphic structure of a family.
	
A natural corollary of Theorem \ref{fiberbime} is an invariance result of plurigenera. This result basically yields that the semi-continuity of plurigenera implies their continuity in certain contexts.
	
\begin{observation}[Invariance of plurigenera]\label{fiberbime global-cor}
With Notation \ref{notation 1}, we further assume that fibers of both $\pi_1$ and $\pi_2$ are irreducible. We also assume that the Kodaira dimension of $X_t$ is non-negative for any $t \in S$. 
Moreover, for any $s \in S$, we assume:
\begin{enumerate}
\item
For any surjective morphism $\mathcal{Z} \to \mathcal{X}$ over $\Delta$ such that the fiber $Z_t$ of $\mathcal{Z} \to \Delta$ over $t \in \Delta$ is a proper modification of $X_t$ for $t$ near $s$, any smooth divisor (not necessarily connected) $Z^{\prime}$ contained in $Z_s$ satisfies $P_m(Z^{\prime}) \le P_m(Z_t)$ for any $m \in \mathbb{N}^+$ and $t$ near $s$, where $\Delta$ is an open neighborhood of $s$ whose local model is the unit disk in $\mathbb{C}$. 
\item 
 $P_m(Y_s) \geq P_m(Y_t)$ for any $t$ near $s$. 
 \end{enumerate}
 Then $P_m(X_t) = P_m(Y_t)$ is independent of $t \in S$ for any positive integer $m$.
\end{observation}
		
\subsection{Refinements or complements of Theorem \ref{fiberbime}} \label{ss refinement}
From the proof of Theorem \ref{fiberbime}, it is clear that the conditions for the sequence $\{P_m\}_{m \in \mathbb{N}}$ can be relaxed, as demonstrated by the following three results. Notably, the genus conditions in these results involve only a single \(P_{\ell}\) for any fixed \(\ell \in \mathbb{N}\). In particular, Theorem \ref{a genus case} roughly provides the relationship diagram:
\begin{center}
		deformation behavior of \(1\)-genus \(\Rightarrow\) fiberwise bimeromorphism.
\end{center}
	
Since the proofs of the next three results closely resemble that of Theorem \ref{fiberbime}, we will omit the detailed explanations and instead provide a proof sketch at the end of this subsection.
	
\begin{thm} \label{a genus case}
With Notation \ref{notation 1}, assume that for some \(s \in S\) there exists an \(\ell \in \mathbb{Z}^+\) (depending on \(s\)) with the properties:
\begin{enumerate}[\rm{(}1\rm{)}]
			\item \label{ellgenus-fiberbime1-global}\(P_{\ell}(X_s) \geq 1\);
			%There is a component of \(X_s\) whose \(\ell\)-genus \(P_{\ell} \geq 1\);
			\item \label{ellgenus-fiberbime2-global} {Lower semi-continuity:} For any semi-stable reduction \(\mathcal{Z} \to \mathcal{X}\) over an open neighborhood \(\Delta\) of \(s\), with the local model being the unit disk in \(\mathbb{C}\), followed by some blow-ups, such that the fiber \(Z_t\) of \(\mathcal{Z} \to \Delta\) over \(t \in \Delta\) is a proper modification of \(X_t\) for \(t\)  $(\neq s)$ near \(s\), any smooth divisor \(Z'\) (not necessarily connected) contained in \(Z_s\) satisfies \(P_{\ell}(Z') \leq P_{\ell}(Z_t)\) for any \(t\) near \(s\);
			\item \label{ellgenus-fiberbime4-global}{Upper semi-continuity:} There exists a reduced and irreducible component \(D_s\) of \((Y_s)_{\textrm{red}}\) with \(P_{\ell}(D_s) \geq P_{\ell}(Y_t)\) for \(t\) $(\neq s)$ near \(s\).
		\end{enumerate}
		Then, we have the conclusions:
		\begin{enumerate}[\rm{(}i\rm{)}]
			\item There exists an irreducible component \(C_s\) of \((X_s)_{\textrm{red}}\) such that \(P_{\ell}(C_s) \geq 1\) and \(C_s\) is bimeromorphic to \(D_s\), induced by \(f\);
			\item Any other irreducible component \(C'_s\) of \((X_s)_{\textrm{red}}\) cannot satisfy \(P_{\ell}(C'_s) \geq 1\). In particular,  $C_s$ is the unique irreducible component that is bimeromorphic to $D_s$;
			\item Any other irreducible component \(D'_s\) of \((Y_s)_{\textrm{red}}\) cannot satisfy the condition \eqref{ellgenus-fiberbime4-global} for \(D_s\).
\end{enumerate}
In particular, if each fiber of \(\pi_1\) and \(\pi_2\) is  irreducible and each point of \(S\) 
satisfies the conditions identical to \eqref{ellgenus-fiberbime1-global} \eqref{ellgenus-fiberbime2-global} \eqref{ellgenus-fiberbime4-global}  for $s\in S$, then \(f\) is fiberwise bimeromorphic in the sense of Definition \ref{fiberbimedef}.
\end{thm}
	
\begin{proof}
One only needs to replace \(m\) in \eqref{contradiction-0801} (or \eqref{contradiction ine}) by \(\ell\) in the condition of Theorem \ref{a genus case}, and one can still derive that \(P_{\ell}(Z) = 0\), which contradicts the condition that there is a component of \(X_0\) whose \(\ell\)-genus \(P_{\ell} \geq 1\).	
\end{proof}
\begin{rem}
By Theorem \ref{a genus case}, the deformation behavior of \(1\)-genus can help characterize the fiberwise bimeromorphism. As extending canonical forms is generally easier than extending pluricanonical forms, Theorem \ref{a genus case} has greater potential for application than Theorem \ref{fiberbime}.
\end{rem}
	
\section{Proof of Theorem \ref{fiberbime}}\label{direct proof-map}

Since a bimeromorphic map between two families over the same base remains a bimeromorphic map over any small connected open subset of the base by Lemma  \ref{restriction of bime}, we can assume from now on that $S$ is the unit disk in $\mathbb{C}$.
Since $\mathcal{Y}$ is irreducible, we take a desingularization $\mathcal{Y}_1 \to \mathcal{Y}$ (e.g., \cite[Theorem 5.4.2]{AHV18} or \cite[Theorem 2.12]{Ue75}) such that $\mathcal{Y}_1$ is a connected complex manifold. Given that the fibers of $\mathcal{Y} \to \Delta$ are of codimension $1$, we assume that this desingularization induces a proper modification on the reduction of each fiber over $t \neq 0$, after shrinking the disk $\Delta$ to a smaller disk $\Delta_1$.  
\footnote{Throughout this proof, the term ``disk" refers to a neighborhood of the origin in $\mathbb{C}$ unless explicitly stated otherwise.}  
	
Denote by $D_1$ the strict transform of $D$ under $\mathcal{Y}_1 \to \mathcal{Y}$, whose well-definedness is ensured by the fact that the codimension of the smooth part of $\mathcal{Y}$ is no less than $2$. Using the generic smoothness theorem, we further shrink $\Delta_1$ to a smaller disk $\Delta_2$ to ensure that the central fiber (over $t=0$) is the only possibly singular fiber of the family $\mathcal{Y}_1 \to \Delta_2$. By performing a series of blow-ups, we obtain a proper modification 
	\[
	\mathcal{Y}_2 \to \mathcal{Y}_1,
	\] 
which induces a modification on each fiber over $t \neq 0$ and ensures that both $\mathcal{Y}_2$ and  the strict transform of the reduction of the  central fiber of $\mathcal{Y}_1 \to \Delta_2$ under $\mathcal{Y}_2 \to \mathcal{Y}_1$ are smooth (e.g., \cite[p. xxiv, Theorem 6]{AHV18}). In particular, the strict transform $D_2$ (whose well-definedness is ensured by \cite[p. xxiv, Theorem 6]{AHV18}) of $D_1$ under $\mathcal{Y}_2 \to \mathcal{Y}_1$ is smooth.
	
Similarly, let $\mathcal{X}_1 \to \mathcal{X}$ be the desingularization of $\mathcal{X}$ such that it induces a proper modification on the reduction of each fiber over $t \neq 0$, after shrinking the disk $\Delta_2$ to a smaller disk $\Delta_3$.  As before, due to the codimension reason of the singular locus of $\mathcal{X}$, the strict transform of $X_0$ is well-defined.
	
Next, we perform a sequence of operations, as illustrated in the  commutative diagram
\begin{center}
\begin{tikzpicture}
\node (X1) at (4,4) {$\mathcal{X}_1$};
\node (Y2) at (8,4) {$\mathcal{Y}_2$};
\node (Deltai) at (6,2) {$\Delta_i$};
\node (X3) at (8,6) {$\mathcal{X}_3$};
\node (X2) at (4,6) {$\mathcal{X}_2$};
\draw[dashed,->] (X1) -- (Y2)  node[midway, below]{$f_1$};
\draw[->] (X1) --    (Deltai) node[midway, above right=1pt]{};
\draw[->] (Y2) -- (Deltai)    node[midway, below right=1pt]{};
\draw[->] (X3) --    (X2) node[midway, below left=1pt]{};	
\draw[->] (X2) --    (X1) node[midway,  left=0.5pt]{$h$};
\draw[->] (X2) --    (Y2) node[midway, above left=0.5pt]{$f_2$};
\draw[->] (X3) --    (Y2) node[midway, below left=1pt]{};	
\end{tikzpicture}.
\end{center}
	
Clearly, there exists a bimeromorphic map $f_1: \mathcal{X}_1 \dashrightarrow \mathcal{Y}_2$. By Lemma \ref{elimination of inde}, we resolve the indeterminacy (its codimension is no less than $2$) of $f_1$ by taking a proper modification $h: \mathcal{X}_2 \to \mathcal{X}_1$, which is a locally finite succession of blow-ups with smooth centers. This gives a composed map $f_2 := f_1 \circ h: \mathcal{X}_2 \to \mathcal{Y}_2$, which is a bimeromorphic morphism.

    It follows from the construction in Lemma \ref{elimination of inde} and the smoothness of   $\mathcal{X}_1$ that $\mathcal{X}_2$ is  also smooth. Furthermore, we assume that $h$ induces a modification on each fiber over $t \neq 0$ after shrinking the disk $\Delta_3$ to a smaller disk $\Delta_4$. By the generic smoothness theorem, after further shrinking $\Delta_4$ to $\Delta_5$, we assume that the central fiber of $\mathcal{X}_2 \to \Delta_5$ is the only possibly singular fiber. Proceeding as before, a finite series of blow-ups gives a proper modification 
	\[
	\mathcal{X}_3 \to \mathcal{X}_2,
	\] 
inducing a proper modification on each fiber over $t \neq 0$ and ensuring that both $\mathcal{X}_3$ and the strict transform (as before, it is well-defined) of the reduction of the central fiber of $\mathcal{X}_2 \to \Delta_5$ under $\mathcal{X}_3 \to \mathcal{X}_2$ are smooth. Clearly, the reduction of the final strict transform $Z$ of $X_0$ under the natural composed morphism $\mathcal{X}_3 \to \mathcal{X}$ is smooth, and the composition $\mathcal{X}_3 \to \mathcal{X}_2 \to \mathcal{Y}_2$ is a bimeromorphic morphism.
	
For convenience, we adopt the notational conventions: $$\mathcal{V} := \mathcal{X}_3, \mathcal{W} := \mathcal{Y}_2, S := \Delta_5, \tilde{D} := D_2,$$ and denote by 
\begin{equation}\label{gVW}
g: \mathcal{V} \to \mathcal{W}
\end{equation}
the composition $\mathcal{X}_3 \to \mathcal{X}_2 \to \mathcal{Y}_2$. Denote by $V_t$ the fiber of $\mathcal{V} \to S$ over $t \in S$, and $W_t$ the fiber of $\mathcal{W} \to S$ over $t \in S$. From the above analysis, $g$ is a bimeromorphic morphism between connected complex manifolds $\mathcal{V}$ and $\mathcal{W}$. The reduction of the final strict transform $Z$ of $X_0$ under $\mathcal{V} \to \mathcal{X}$ is a disjoint union of complex manifolds in $V_0$, while the strict transform $\tilde{D}$ of $D$ under $\mathcal{Y}_2 \to \mathcal{Y}$ is a complex manifold contained in $W_0$. Furthermore, for $t \neq 0$, we have $P_m(V_t) = P_m(X_t)$, $P_m(W_t) = P_m(Y_t)$, $P_m(Z) = P_m(X_0)$ (Definition \ref{reducible genera}), and $P_m(\tilde{D}) = P_m(D)$.
	
\subsection{A quick proof of Theorem \ref{fiberbime}}\label{quick proof}
In this subsection, we provide a concise proof of Theorem \ref{fiberbime}, relying primarily on (the local version of) Hironaka's Chow Lemma \ref{local version of chow}. This lemma asserts that the source and target of a proper modification can be lifted to a common space through respective blow-ups.

Applying  Lemma \ref{local version of chow} to  the proper modification $g$, we obtain that,  after further shrinking the disk to  $S_1$, we can construct a complex analytic space $\mathcal{U}$, which serves as both the blow-up of $\mathcal{V}$ via $h$ and the blow-up of $\mathcal{W}$ via $\ell$
	\begin{center}
		\begin{tikzpicture}
			% Nodes
			\node (V) at (0,2) {$\mathcal{V}$};
			\node (W) at (2,2) {$\mathcal{W}$};
			\node (U) at (1,3) {$\mathcal{U}$};
			% Arrows
			\draw[->] (V) -- (W) node[midway, below]{$g$};
			\draw[->] (U) -- (V) node[midway, above left]{$h$};
			\draw[->] (U) -- (W) node[midway, above right]{$\ell$};
		\end{tikzpicture}.
	\end{center}
	
Since $\mathcal{V}$ and $\mathcal{W}$ are smooth, subvarieties of codimension one in each are Cartier. Thus, we may assume that the centers of $\ell$ and $h$ have codimension at least two. This assumption can also be directly justified by \cite[p. 215]{GR84}.
	
Let the strict transform of the reduction of $Z$ under $h$ be $Z_1 + Z_2 + \cdots + Z_k$, where $k$ is a positive integer and $Z_i$ are irreducible divisors. Denote by $E$ the strict transform of $\tilde{D}$ under $\ell$. Note that both $E$ and $Z_1 + Z_2 + \cdots + Z_k$ are contained in the central fiber $U_0$ of $\mathcal{U} \to S_1$ over $0$.
	
Set $\{Z_j\}_{j \in J}$ as the set of $Z_j$ with $\kappa(Z_j) \geq 0$, where $j \in \{1, \dots, k\}$. Let $j_0$ be an element of $J$. We now present Claim \ref{claim-one element}. If this claim holds, the strict transform of $Z_{j_0}$ under the natural map $\mathcal{X} \dashrightarrow \mathcal{U}$ corresponds to the desired $C$ in \eqref{mainresult1} of Theorem \ref{fiberbime}. Furthermore, this claim also completes the proof of \eqref{mainresult2} of Theorem \ref{fiberbime}, as follows: Condition \eqref{fiberbime1} and the lower semi-continuity condition \eqref{fiberbime2} imply that the Kodaira dimension of the general fiber of $\pi_1$ is no less than $0$,  which then implies $\kappa(D)\geq 0$ by the upper semi-continuity condition \eqref{fiberbime4} and Lemma \ref{obs 1}.
In particular, no irreducible component of $(X_0)_{\rm{red}}$ other than $C$  can be  bimeromorphic to $D$.

\begin{claim}\label{claim-one element}
$Z_{j_0}$ must coincide with $E$. In particular,  $J$ contains only one element and  the strict transform of $Z_{j_0}$ under $h^{-1}$ is bimeromorphic to $\tilde{D}$,  induced by $g$.
\end{claim}
	
\begin{proof}
Assume, for the sake of contradiction, that $E$ is not $Z_{j_0}$. Consider a desingularization $\tilde{\mathcal{U}} \to \mathcal{U}$. Note that the singular locus of $\mathcal{U}$ does not cover any fiber $U_t$ for $t \neq 0$, nor does it cover any of $Z_1, Z_2, \dots, Z_k$, or $E$, due to the codimension of the centers of $h$ and $\ell$. Here, $U_t$ denotes the fiber of $\mathcal{U} \to S_2$ over $t \in S_2$, where $S_2$ is a smaller disk such that no exceptional divisors of $h$ and $\ell$ lie on $U_t$ for any nonzero $t\in S_2$. For convenience, we still denote $\tilde{\mathcal{U}}$ by $\mathcal{U}$, and the strict transforms of $Z_{j_0}$ and $E$ under $\tilde{\mathcal{U}} \to \mathcal{U}$  still by $Z_{j_0}$ and $E$, respectively.
		
For the $1$-parameter degeneration $\mathcal{U}\to  S_2$, we perform a semi-stable reduction followed by certain blowups, and we denote by $b: S_2^{\prime}\to S_2$  the corresponding finite morphism and  $\mathcal{Z}\to S_2^{\prime}$  the resulting new morphism (see \cite[Chapter II]{KKMS73} or \cite[Section 11.2.2]{PS08}).  Note that we can choose the blowups  such that  the resulting morphism 
$\mathcal{Z}  \to S_2^{\prime}$ satisfies:
\begin{itemize}
			\item For any nonzero $t\in S_2^{\prime}$, the fibers $Z_t$ of $\mathcal{Z} \to S_2^{\prime}$ over $t$ are proper modifications of $X_{b(t)}$.
			\item 
        There exists a disjoint union $Z^{\prime}$ of smooth irreducible components of $Z_0$ that dominates the base change of $Z_{j_0}+E$ induced by $b$. In particular, $P_m(Z^{\prime})\geq P_m(Z_{j_0})+P_m(E)$, based on the basic fact that base change does not change the fiber (e.g., \cite[p. 29]{Fs76}).      
\end{itemize}
Then by Condition \eqref{fiberbime2} of Theorem \ref{fiberbime}, we conclude that $$P_m(E) + P_m(Z_{j_0}) \leq P_m(X_t)$$ for any positive integer $m$ and for any general $t\in S_2$. Consequently, for any nonzero $t\in S_2$ near $0$, 
\begin{equation}\label{contradiction-0801}
P_m(E) + P_m(Z_{j_0}) \leq P_m(X_t) = P_m(U_t) = P_m(Y_t) \leq P_m(D) = P_m(E),
\end{equation}
where $P_m(Y_t) \leq P_m(D)$ is the given condition \eqref{fiberbime4}.	Hence, $P_m(Z_{j_0}) = 0$ for all positive integers $m$, contradicting the assumption $\kappa(Z_{j_0}) \geq 0$. Therefore, $Z_{j_0}$ must coincide with $E$. This completes the proof of this claim.
\end{proof}
	
For any other irreducible component $D' \neq D$ of $(Y_0)_{\mathrm{red}}$ with \eqref{fiberbime4} in Theorem \ref{fiberbime}, its strict transform under the bimeromorphic morphism $\mathcal{U} \to \mathcal{W} \to \mathcal{Y}_1 \to \mathcal{Y}$ must also be the unique $Z_{j_0}$. This contradicts the basic fact that $D'$ and $D$ have distinct strict transforms under $\mathcal{U} \to \mathcal{W} \to \mathcal{Y}_1 \to \mathcal{Y}$. The proof of \eqref{mainresult3} in Theorem \ref{fiberbime} is now complete, and Theorem \ref{fiberbime} is thus proved.
	
\subsection{An elementary approach for the existence of \(C\) in \eqref{mainresult1} of Theorem \ref{fiberbime}}\label{long proof}
	
The quick proof presented in the preceding subsection relies primarily on the strength of the local version of the Chow lemma. In this subsection, we propose an alternative proof for the existence of \(C\) in \eqref{mainresult1} of Theorem \ref{fiberbime}. This proof is, in some sense, more straightforward and elementary (it relies on the argument of Theorem \ref{new-fiberbime-irredu-singular}, which essentially stems from an elementary analysis of the ramification divisor), and it is of independent interest, as it illustrates how to rule out exceptional divisors.
	
The first part of the proof is identical to the portion between the titles of Section \ref{direct proof-map} and Subsection \ref{quick proof}, and we adopt the same notations. The remaining part is devoted to proving the existence of \(C\) in \eqref{mainresult1} of Theorem \ref{fiberbime} and divided into two steps.

\noindent
\textbf{Step 1.} {Exceptional divisors of \(\mathcal{V} \to \mathcal{X}\) cannot be mapped by \(g\) onto \(\tilde{D}\).}
	
Now we are ready to identify a component of \(Z\) which may be bimeromorphic to \(\tilde{D}\). 
\begin{prop}\label{ramZ} The 
exceptional divisors of $g$ in \eqref{gVW} cannot be mapped by \(g\) onto \(\tilde{D}\), and there exists some irreducible component \(E\) of \(Z\) such that \(g(E) = \tilde{D}\).
\end{prop}
\begin{proof}
Recall that \(g: \mathcal{V} \to \mathcal{W}\) is a bimeromorphic morphism from a connected complex manifold \(\mathcal{V}\) to a connected complex manifold \(\mathcal{W}\). Then it is surjective by definition. Note that \(g|_{V_0}\) is proper due to the compactness of \(V_0\), and thus it sends any analytic subset of \(V_0\) to an analytic subset of \(W_0\) by Remmert's proper mapping theorem.  It follows that there exists an irreducible component \(E\) (possibly an exceptional divisor of \(\mathcal{V} \to \mathcal{X}\)) of the reduction of \(V_0\) such that \(g(E) = \tilde{D}\) by the irreducibility of \(\tilde{D}\).
		
So, for any positive integer \(m\), 
$$P_m(E) \geq P_m(\tilde{D}) = P_m(D)$$  
by \cite[Lemma 6.3]{Ue75}. Recall that Lemma \ref{obs 1} says that \(V_t\) is bimeromorphic to \(W_t\) for any general \(t \in S\), implying $$P_m(X_t) = P_m(V_t) = P_m(W_t) = P_m(Y_t)$$ for all \(t \neq 0\), after shrinking \(S\) to some smaller disk \(S_1\).
		
Assuming by contradiction that \(E\) is not an irreducible component of \(Z_{\mathrm{red}}\), we perform a semi-stable reduction \(\mathcal{Z} \to \mathcal{V}\) \footnote{To save notation, we do not distinguish between a complex space and its finite base change in the process of the semi-stable reduction. For precise details, we refer the reader to paragraph 2 of the proof of Claim \ref{claim-one element}.} such that the strict transform \(Z'\) in \(\mathcal{Z}\) of the divisor \(E + Z\) is smooth. By repeating the argument in Subsection \ref{quick proof} and using Definition \ref{reducible genera}, Conditions \eqref{fiberbime2}, \eqref{fiberbime4} in Theorem \ref{fiberbime}, and the inequality \(P_m(E) \geq P_m(\tilde{D}) = P_m(D)\), we derive
\begin{equation}\label{contradiction ine}
P_m(Z) + P_m(D) \leq P_m(Z) + P_m(E) \leq P_m(Z_t) = P_m(Y_t) \leq P_m(D).
\end{equation}
Thus, \(P_m(X_0) = P_m(Z) = 0\), which contradicts Condition \eqref{fiberbime1} by the arbitrariness of \(m \in \mathbb{N}^+\). Therefore, \(E\) must be an irreducible component of \(Z\), completing the proof of Proposition \ref{ramZ}.
\end{proof}

\noindent
\textbf{Step 2.} {The existence of \(C\) in \eqref{mainresult1} of Theorem \ref{fiberbime}.}

Recall that $g: \mathcal{V} \to \mathcal{W}$ is a bimeromorphic morphism between connected complex manifolds $\mathcal{V}$ and $\mathcal{W}$, and that the natural morphism $ \mathcal{W}\to S$ is flat by the criterion (e.g., \cite[Paragraph 1 of Introduction]{Hi75}, \cite[Lemma 2.1]{WZ23}) for flatness of a morphism over a smooth curve. 
Then we can apply the  argument,  similar to that of Theorem \ref{new-fiberbime-irredu-singular}, to verify that \(E\) is bimeromorphic to \(\tilde{D}\). 
 
In fact, by \cite[(1.17) + Theorem 1.19]{PR94}, there exists an analytic subset \(\Lambda \subset \mathcal{W}\) with codimension at least \(2\), such that \(\dim_x V_w = 0\) for any \(x \in V_w\) and \(w \in \mathcal{W} \setminus \Lambda\), where \(V_w := g^{-1}(w)\). Since \(g\) is a proper modification, each fiber of \(g\) is connected, and $$g|_{g^{-1}(\mathcal{W} \setminus \Lambda)}: g^{-1}(\mathcal{W} \setminus \Lambda) \to \mathcal{W} \setminus \Lambda$$ is biholomorphic by \cite[p. 166]{GR84}.
Since \(\tilde{D} \not\subset \Lambda\), \(E \not\subset g^{-1}(\Lambda)\), \(g: E \to \tilde{D}\) is bimeromorphic.

Consequently, the strict transform of \(E\) under \(\mathcal{X} \to \mathcal{V}\) is the desired \(C\) in \eqref{mainresult1} of Theorem \ref{fiberbime}.  This completes the proof of the existence of \(C\) in \eqref{mainresult1} of Theorem \ref{fiberbime}.
	
\section{Fiberwise bimeromorphism between locally Moishezon families}\label{section-moishezon}
	
As applications of (the argument of) Theorem \ref{fiberbime}, we obtain several results on fiberwise bimeromorphism of a bimeromorphic map between Moishezon families over the same base.
\subsection{Extension and invariance results for Moishezon families}
We will recall Takayama's pluricanonical extension and invariance of plurigenera for Moishezon families. Canonical singularities play an important role there. 
	
\begin{defn}\label{def-cano singu}
A normal complex analytic space \(X\) has only \emph{canonical singularities} if
\begin{enumerate}
\item \(X\) is \(\mathbb{Q}\)-Gorenstein;
\item For any resolution \(f: W \rightarrow X\), if \(K_W = f^*K_X + E\) with \(E\) an \(f\)-exceptional divisor, then \(E \geq 0\). 
Here $K_\bullet$ are the canonical divisors. \end{enumerate}
\end{defn}
Takayama extends pluricanonical sections from the union of some irreducible components of a fiber to the total space.
\begin{lemma}[{\cite[Theorem 3.1]{Tk07}}]\label{Taka-m-extension}
Let \(X\) be a complex manifold of dimension \(n\), \(S = \{|t| < 1\}\) the unit disk in \(\mathbb{C}\), and \(\pi: X \to S\) a projective surjective morphism with connected fibers. Assume that the prime decomposition of  \(X_0 := \pi^{-1}(0)\) is 
\[
X_0 = \sum_{i \in I} X_i + \sum_{j \in J} k_j X_j',
\]
such that \(Y := \sum_{i \in I} X_i\) is a disjoint union of smooth divisors. In particular, each element in \(\{X_i\}_{i \in I}\) is distinct from every element in \(\{X_j'\}_{j \in J}\).
		
Additionally, let \(\ell\) be a positive integer, and \(s^{(\ell)} \in H^0(Y, \ell K_Y) = \bigoplus_{i \in I} H^0(X_i, \ell K_{X_i})\). Then \(s^{(\ell)} \wedge (dt)^\ell \in H^0(Y, \left.\ell K_X\right|_Y)\) can be extended to an element of \(H^0(X, \ell K_X)\). In particular, 
\[
\sum_{i \in I} P_\ell(X_i) \leq P_\ell(X_t)
\]
holds for a general fiber \(X_t\), where \(P_\ell\) denotes the usual \(\ell\)-genus.
\end{lemma}
	
Takayama's semi-continuity and invariance results will also play a crucial role in our argument.
\begin{lemma}[{\cite[Theorem 1.2]{Tk07}}]\label{Taka lsc-preliminary}
Let \(\pi: \mathcal{X} \rightarrow S\) be a locally Moishezon family,  and let \(X_s := \pi^{-1}(s)\) be a reference fiber with support \(\left(X_s\right)_{\textrm{red}} = \sum_{i \in I} X_i\), and \(X_t := \pi^{-1}(t)\) a general fiber. Then
\[
\sum_{i \in I} P_m(X_i) \leq P_m(X_t)
\]
holds for any positive integer \(m\).
\end{lemma}
	
\begin{lemma}[{\cite[Theorem 1.1]{Tk07}}]\label{Taka inva-preliminary}
Let \(\pi: \mathcal{X} \rightarrow S\) be a locally Moishezon family (Definition \ref{def-Moish}) such that every fiber \(X_t := \pi^{-1}(t)\) has only canonical singularities (Definition \ref{def-cano singu}). Then the \(m\)-genus \(P_m(X_t)\) is independent of \(t \in S\) for any positive integer \(m\).
\end{lemma}
	
\subsection{Fiberwise bimeromorphism for Moishezon families}
As an application of (the argument of) Theorem \ref{fiberbime}, we obtain several results on fiberwise bimeromorphism of a bimeromorphic map between Moishezon families over the same base.
\begin{thm}\label{fiberbime taka mois}
With Notation \ref{notation 1}, assume that $\pi_1$ is locally Moishezon. Furthermore, for some $s \in S$, $\kappa(X_s) \geq 0$ (see Definition \ref{reducible genera}), and $D_s$ is an irreducible component of $(Y_s)_{\textrm{red}}$ such that for any $m \in \mathbb{N}^+$, $P_m(D_s) \geq P_m(Y_t)$ (hence $P_m(D_s) = P_m(Y_t)$ and the $m$-genus of any other irreducible component of $(Y_s)_{\textrm{red}}$ is $0$, based on Lemma \ref{Taka lsc-preliminary}) for any $t$ $(\neq s)$ near $s$. Then we have:
\begin{enumerate}[\rm{(}i\rm{)}]
			\item\label{moishezon-1} 
			There exists an irreducible component $C_s$ of $(X_s)_{\textrm{red}}$ such that $\kappa(C_s) \geq 0$ and $C_s$ is bimeromorphic to $D_s$, induced by $f$;
			\item\label{moishezon-2}  
			For any other irreducible component $C^{\prime}_s$ of $(X_s)_{\textrm{red}}$, $\kappa(C^{\prime}_s) = -\infty$. In particular,  $C_s$ is the unique irreducible component that is bimeromorphic to $D_s$.
\end{enumerate}  
In particular, if for any $t \in S$, both $X_t$ and $Y_t$ are irreducible, $\kappa(X_t) \geq 0$ and $(Y_t)_{\textrm{red}}$ satisfies the condition for $D_s$, then  $f$ is fiberwise bimeromorphic in the sense of Definition \ref{fiberbimedef}. 
\end{thm}

\begin{proof}
We can quickly complete this proof by verifying the conditions of Theorem \ref{fiberbime} using Lemma \ref{Taka-m-extension} or \ref{Taka lsc-preliminary}. However, for a better understanding of the argument of Theorem \ref{fiberbime}, particularly of the role of Condition \eqref{fiberbime2} ($\pi_1$ in this theorem is a good example satisfying the condition \eqref{fiberbime2} in Theorem \ref{fiberbime}) in Theorem \ref{fiberbime}, we now provide a more detailed verification of the fiberwise bimeromorphism, following the proof of Theorem \ref{fiberbime}. This approach is, more precisely, based on Lemma \ref{Taka-m-extension} or Lemma \ref{Taka lsc-preliminary} and the arguments of Theorem \ref{fiberbime}.
		
We first desingularize both $\mathcal{X}$ and $\mathcal{Y}$ to make them smooth, as done in the first paragraph of Section \ref{direct proof-map}. For convenience, we still denote their desingularization by $\mathcal{X}$ and $\mathcal{Y}$, respectively. Additionally, the strict transform of $D_s$ is also denoted by $D_s$, and the strict transform of $X_s$ is still denoted by $X_s$.
		
We now perform a series of operations for the bimeromorphic map $f$, as shown in the  
 commutative diagram:  
\begin{center}
			\begin{tikzpicture}
				% 绘制交换节点
				\node (X) at (4,4) {$\mathcal{X}$};
				\node (Y) at (8,4) {$\mathcal{Y}$};
				\node (S) at (6,2) {S};
				\node (X1) at (4,6) {$\mathcal{X}_1$};
				\node (U) at (8,6) {$\mathcal{U}$};
				% 绘制箭头
				\draw[dashed,->] (X) -- (Y) node[midway, above]{$f$};
				\draw[->] (X) -- (S) node[midway, above right=1pt]{};
				\draw[->] (Y) -- (S) node[midway, below right=1pt]{};
				\draw[->] (X1) -- (X) node[midway, above left=0.5pt]{$\mu$};
				\draw[->] (X1) -- (Y) node[midway, above left=1pt]{$g$};    
				\draw[->] (U) -- (X1) node[midway, above left=1pt]{$\theta_1$};    
				\draw[->] (U) -- (Y) node[midway, above left=1pt]{$\theta_2$};                
			\end{tikzpicture}.
\end{center}  
		
\begin{enumerate}
			\item 
By Lemma \ref{elimination of inde}, we eliminate the indeterminacy of $f$ such that $\mu$ is a composition of a finite succession of blow-ups with smooth centers and $g$ is a proper modification. The strict transform of $X_s$ under $\mu$ is denoted by $X_s^{\prime}$. Note that both $\mathcal{X}_1$ and $\mathcal{Y}$ are smooth;
			\item
By Lemma \ref{local version of chow}, after shrinking the base, we can find a common complex analytic space $\mathcal{U}$ that dominates both $\mathcal{X}_1$ and $\mathcal{Y}$ via the blow-ups $\theta_1$ and $\theta_2$, respectively.      
\end{enumerate}  
		
Note that we can always assume that both the centers of $\theta_1$ and $\theta_2$ are of codimension at least $2$, based on the same reason as explained in the third paragraph of Subsection \ref{quick proof}. Then the strict transform (denoted by $\tilde{D}_s$) of $D_s$ under $\theta_2$, and the strict transform (whose reduction is denoted by $Z_1 + Z_2 + \cdots + Z_k$ for some positive integer $k$ and irreducible divisors $Z_i$) of $X_s^{\prime}$ under $\theta_1$ are located in $U_s$, the fiber of the natural composed morphism $\mathcal{U} \to S$ over $s$.
		
Let $\{Z_j\}_{j \in J}$ be the set of $Z_j$ with $\kappa(Z_j) \geq 0$ with $j \in \{1, \cdots, k\}$. Let $j_0$ be an element of $J$. We now present:
\begin{claim}\label{6.6}
$Z_{j_0}$ must be $\tilde{D}_s$. In particular,  $J$ has only one element and  the strict transform of $Z_{j_0}$ under ${\theta_1}^{-1}$ is bimeromorphic to ${D}_s$, induced by $g$. 
\end{claim}
\begin{proof}
{Assume by contradiction that  $\tilde{D}_s$ is not $Z_{j_0}$.
Take a resolution $\nu: \mathcal{W}\to \mathcal{U}$  whose center is exactly the singular locus of $\mathcal{U}$.
Note that  the singular locus of $\mathcal{U}$  covers neither  $\tilde{D}_s$ nor  $Z_{j_0}$. Then both of the strict transform  under $\nu$ of $\tilde{D}_s$ and that of  $Z_{j_0}$ are well-defined and they are two distinct irreducible components of the fiber of the natural composed morphism $\mathcal{W}\to S$ over $s$.}

{Clearly, it follows from  Lemma \ref{fiber connected after modification} that  the natural composed morphism $\mathcal{W}\to S$ is a locally Moishezon family. } Then 
 Lemma \ref{Taka lsc-preliminary} \footnote{Note that  the  total space  in \cite[Theorem 1.2]{Tk07} is not  assumed to be normal. Hence we can  directly apply \cite[Theorem 1.2]{Tk07} to the locally Moishezon morphism $\mathcal{U}\to S$ (since $\mathcal{U}$ may not be normal, $\mathcal{U}\to S$ may not be called a family in the present paper).} gives rise to $$P_m(Z_{j_0}) + P_m(\tilde{D}_s) \leq P_m(U_t) = P_m(Y_t) \leq P_m(D_s)$$ for any positive integer $m$ and $t \neq s$, where the first ``$\leq$" is due to Lemma \ref{Taka-m-extension} or \ref{Taka lsc-preliminary}. 
Therefore, $P_m(Z_{j_0}) = 0$ for any positive integer $m$, which contradicts the condition that $\kappa(Z_{j_0}) \geq 0$.
\end{proof}   

With this claim in mind, the strict transform of $Z_{j_0}$ under the natural bimeromorphic map $\mathcal{X} \dashrightarrow \mathcal{U}$ is exactly the desired $C_s$ in \eqref{moishezon-1} of Theorem \ref{fiberbime taka mois}. Consequently, the proof of Theorem \ref{fiberbime taka mois} is complete.
\end{proof}  

As a direct application of Theorem \ref{fiberbime taka mois} and  Lemma \ref{Taka inva-preliminary}, we can easily obtain: 	
\begin{cor}\label{fibe bime canonical singu-kodaira geq 0}
With Notation \ref{notation 1}, we assume that $\pi_1$ is locally Moishezon and each fiber $X_t$ of $\pi_1$ satisfies $\kappa(X_t) \geq 0$ for any $t \in S$. Moreover, each fiber of $\pi_2$ has only canonical singularities (Definition \ref{def-cano singu}). Then there exists a unique irreducible component of $(X_s)_{\textrm{red}}$ that is bimeromorphic to $Y_s$ (automatically reduced) for any $s \in S$. In particular, if we further assume that each fiber of $\pi_1$ is irreducible, then $(X_t)_{\textrm{red}}$ is bimeromorphic to $Y_t$, induced by $f$, for any $t \in S$.
\end{cor}

An interesting example of Corollary \ref{fibe bime canonical singu-kodaira geq 0} on fiberwise bimeromorphism is provided by \cite[p. 352, Last paragraph of the proof for Theorem 1]{Kol21}  (or \cite[Theorem 1.37]{Kol23} for its algebraic analogue). It states that if there is a projective fiber $X_o$ of general type and with canonical singularities in a flat family $\sigma:\mathcal{X}\rightarrow S$ of complex analytic spaces, then $\sigma$ is fiberwise bimeromorphic to the relative canonical model ${Proj}_S\bigoplus_{i\ge 0}\sigma_*^m\omega_{{\mathcal{X}^m}/S}^{[r]}$, which is projective over $S$ (possibly shrunk around $o$) and where $\omega_{{\mathcal{X}^m}/S}^{[r]}$ denotes the double dual of the $r$-th tensor power of $\omega_{{\mathcal{X}^m}/S}$. Since $X_o$ is of general type, a suitable
MMP for $X_o$ ends with a minimal model ${X}_o^m$, and then $X_o\dashrightarrow X_o^m$ extends to a fiberwise bimeromorphic map $\mathcal{X}\dashrightarrow \mathcal{X}^m$, and we have $\sigma^m:  \mathcal{X}^m\rightarrow S$.

{Note  that, in the setting of Theorem \ref{fiberbime taka mois}, for any proper modification $\mathcal{X}'\to \mathcal{X}$    with $\mathcal{X}'$ normal and any irreducible exceptional divisor $E$ contained in the fiber of $\mathcal{X}'\to S$ over the $s\in S$ in Theorem \ref{fiberbime taka mois}, one has $\kappa(E)=-\infty$. In particular, as a direct consequence of Theorem \ref{fiberbime taka mois} and Lemma \ref{Taka inva-preliminary}, we obtain Corollary \ref{exc divisor kappa negative}. This stands in sharp contrast to \cite[Proposition 3.3]{HK15}, where it is shown that an exceptional divisor of a general blow-up may have arbitrary Kodaira dimension.}

\begin{cor}\label{exc divisor kappa negative}
{Let $\mathcal{X}\to S$ be 
 a locally Moishezon family whose fibers all have nonnegative Kodaira dimension and only canonical singularities. Then  for any proper modification $\mu:\mathcal{X}'\to \mathcal{X}$  with $\mathcal{X}'$ normal   and any  irreducible exceptional divisor $F$ lying in certain fiber of $\mathcal{X}'\to S$, we have $\kappa(F)=-\infty$.     }
\end{cor}

\begin{proof}
{Applying Theorem \ref{fiberbime taka mois} to the bimeromorphic  map $\mu:\mathcal{X}'\to \mathcal{X}$ over $S$, yields the desired conclusion, based on  Lemma \ref{Taka inva-preliminary}.  In fact, assume that $F$ lies in the fiber of $\mathcal{X}'\to S$ over $s\in S$, then for a small open neighborhood $U$ of $s$, $\mathcal{X}'\to S$ is also locally Moishezon, based on the local version of Chow lemma. 
By Lemma \ref{Taka inva-preliminary},  $\mathcal{X}\to S$ satisfies the condition for $\pi_2$ in Theorem \ref{fiberbime taka mois}. By the fact that the center of $\mu$ is of codimension at least $2$ (e.g., \cite[p. 215]{GR84}),  $\mathcal{X}'\to S$ satisfies the condition for $\pi_1$ in Theorem \ref{fiberbime taka mois}. It then follows from  Theorem \ref{fiberbime taka mois} that $\kappa(F)=-\infty$.} 
\end{proof}

Furthermore, another straightforward application of Theorem \ref{fiberbime taka mois} and Lemma \ref{Taka inva-preliminary}, combined with
Corollary \ref{smoothcase uniruled}, leads to the following result.

\begin{cor}
Let $p:\mathcal X \to S$ be a locally Moishezon family with \(\mathcal X\)  smooth.
Assume  that there exists a family \(q:\mathcal Y \to S\)
such that \(\mathcal X\) is bimeromorphic to \(\mathcal Y\) over \(S\), and that for every
\(s\in S\), the fiber \(Y_s\) has only canonical singularities and satisfies $\kappa(Y_s)\geq 0.$
Then  every fiber of
\(p:\mathcal X\to S\) contains  exactly one irreducible component of
non-negative Kodaira dimension.
\end{cor}

\begin{proof}
 It follows from Corollary \ref{smoothcase uniruled} that    there exists an irreducible component  bimeromorphic to $Y_s$, which implies  $\kappa(X_s)\geq 0$ for each $s\in S$. This observation then follows directly from
 Theorem \ref{fiberbime taka mois}, based on Lemma \ref{Taka inva-preliminary}. 
\end{proof}

\section{Specialization of   bimeromorphic types for locally Moishezon families}\label{s defor limit}
As is well-known, specialization of rationality and other related concepts (such as stable rationality, unirationality, etc.) are classical and significant topics in algebraic geometry in recent years. K. Timmerscheidt \cite[Theorem 1]{T82}  made use of the relative Barlet cycle space theory and his characterization of $3$-dimensional rational manifolds to obtain that the specialization of rational manifolds is still rational in a smooth family of Moishezon $3$-folds. 

Recently, there has been rapid progress in establishing such results in various settings. 
Among them, our work is largely inspired by the seminal contributions of Kontsevich--Tschinkel \cite[Theorem 1+Theorem 16]{KT19} and Nicaise--Ottem \cite[Theorem 4.1.1]{NO21}. These works fully address the specialization of rationality for smooth families (or the pair $(\text{total space}, \text{special fiber})$ has $B$-rational singularities) with arbitrary-dimensional fibers, in the scheme category. In fact, they have achieved even more in the case of smooth families: the specialization of (or geometric) birational types (Theorem \ref{introduction-KT}).
Note that both  \cite{KT19} and \cite{NO21} dealt with smooth families in the scheme category and the proofs in both works relied on the existence of some specialization morphism. However, in an analytic setting, it seems that the lack of good algebraic structure makes it difficult to construct a similar specialization morphism. 

In this section, we use our works on fiberwise bimeromorphism to study the problem on specialization of bimeromorphic types for families (not necessarily smooth) with fiber not being schemes, in the case where each fiber of involved families has Kodaira dimension not equal to $-\infty$, which can be viewed as an analytic version of the specialization of birational types in  \cite{KT19} and  \cite{NO21}.  Note that in this section, we focus on locally Moishezon families (whose fibers are Moishezon and correspond to algebraic spaces).  In our subsequent paper, we will establish several results on the specialization of bimeromorphic types for certain non-smooth locally Fujiki families, which are defined as families that are bimeromorphic to locally K\"ahler families over the same base.

\subsection{Relative Barlet cycle space theory}\label{ss7.1}
We first provide some background on the relative Barlet cycle space theory, and we refer the reader to \cite{B75}, \cite{Mn04}, \cite{BM19},  \cite{BM25}, \cite{Fu78/79} and \cite[Sections 2 and 3]{CP94} or \cite{Kol96} (for Chow variety theory  over a field of characteristic zero) for more details.

{Let $X$ be a complex space and $n \in \mathbb{N}$ an integer. A \emph{(compact) $n$-cycle of $X$} is a finite linear combination $Z=\sum_{i \in I} n_i Z_i$ where the $Z_i$'s are $n$-dimensional compact irreducible analytic  subsets of $X$ which are pairwise distinct. 
The set of all $n$-cycles of $X$ is denoted by $\mathcal{C}_n(X)$, and the set of all cycles of $X$ is the union of all $\mathcal{C}_n(X)$ for $n \in \mathbb{N}$, denoted by $\mathcal{C}(X)$.   D. Barlet  introduced a reduced analytic structure on $\mathcal{C}(X)$.
We then call $\mathcal{C}(X)$ the \emph{Barlet cycle space or cycle space} of $X$. 
Note that $\mathcal{C}(X)$ is countable at infinity if so is $X$. In particular, for the total space  $X$ of any family in the present paper, $\mathcal{C}(X)$ has at most countably many irreducible components.  Additionally, all the reduced and irreducible cycles of $X$ form an analytic Zariski-open subset $\mathcal{C}^*(X)$ of $\mathcal{C}(X)$ (e.g., \cite[Proposition 4.7.2]{BM19}).   For any subset $W$ of a cycle space, we write $W\cap \mathcal{C}^*$ to denote the set of reduced and irreducible cycles in $W$.}

{From a functorial perspective, $\mathcal{C}_n(X)$ is a reduced complex space that represents
  a contravariant functor $F_X^n$ from the category of  reduced complex spaces  to the category of sets, associating to each reduced complex space $S$ the set of analytic families of $n$-cycles of $X$ parametrized by $S$ (e.g., \cite[Corollary 5.8.2.6]{BM25},  \cite[Section 2.6]{CP94}). In particular,  a morphism $f\in \operatorname{Hom}(S, \mathcal{C}_n(X))$ from  $S$ to $\mathcal{C}_n(X)$ corresponds to such an analytic family (e.g., \cite[Section 4.6.1]{BM19}).  As a result, every such $f$ induces an analytic subset of $X\times S$ that is proper over $S$. }    

\begin{lemma}[{e.g., \cite[Lemma 3.1]{Fu78/79}}]\label{analytic subset}
Let $S$ be a reduced (the non-reduced case can also be treated, via the reduction map) complex analytic space and $f: S\to \mathcal{C}_n(X)$ a holomorphic map. Then
$$X_S:=\{(x,s)\in X\times S: x\in |f(s)|\}$$ 
is an analytic subset of $ X\times S$  and  the morphism $X_S\to S$ induced by the natural projection  $X\times S\to S$ is proper, where $|f(s)|$ is the support of the cycle $f(s)$.  
\end{lemma}	
	
{Following Barlet's cycle theory \cite{B75}  in the absolute setting, Fujiki  \cite[Section 3.2]{Fu78/79} investigated the relative $n$-cycle space $\mathcal{C}_n(X/S)$ for a proper holomorphic map $f: X\to S$ of complex analytic spaces.}
Set-theoretically, $\mathcal{C}_n(X/S)$ is the subset of $\mathcal{C}_n(X)$ consisting of those $n$-cycles of $X$ whose support is contained in one fiber of $X\to S$;   $\mathcal{C}_n(X/S)$  is an  analytic subset of $\mathcal{C}_n(X)$  (e.g., \cite[p. 493, Theorem 4.8.4]{BM19}); 
$\mathcal{C}_n(X/S)$ is equipped with the reduced structure. Furthermore, the natural map $$\varpi:\mathcal{C}_n(X/S)\to S$$ is holomorphic. 

{Recall from \cite[Lemma 4.2.10]{BM19} that for any open subset $U\subseteq S$ of $S$, the natural inclusion $\mathcal{C}_n\left(f^{-1}(U)\right) \to \mathcal{C}_n(X)$ induces a homeomorphism of $\mathcal{C}_n\left(f^{-1}(U)\right)$ onto the open set
$$
\left\{Z \in \mathcal{C}_n(X):|Z| \subseteq f^{-1}(U)\right\}.
$$
Note that $\mathcal{C}(X/S)$  carries the subspace topology induced from $\mathcal{C}(X)$ (e.g., \cite[Section 3.2]{Fu78/79}). Then the restriction of the above homeomorphism induces  a homeomorphism from $\mathcal{C}_n(f^{-1}(U)/U)$  to the open subset 
$$\varpi^{-1}(U):=\mathcal{C}_n(X/S)\times_S U$$
of $\mathcal{C}_n(X/S)$.}

{Functorially (\cite[Proposition   3.2]{Fu78/79}), the relative  $n$-cycle space $\mathcal{C}_n(X/S)$  is an object  in the category $\left(\mathcal{A} n_{\text {red }} / S\right)$ of reduced complex spaces over $S$ that represents  a contravariant functor
$F_{S, n}:\left(\mathcal{A} n_{\text {red }} / S\right)$ → (Sets).   This 
 functor is defined by 
\[ F_{S, n}(T): = \text{the set of analytic families of compact } n\text{-cycles } of X / S \text{ parametrized by } T. \]  
Note that for any reduced complex space $T$ over $S$ with the structure morphism $\alpha: T \to S$, an \emph{analytic family of compact $n$-cycles of $X / S$ parametrized by $T$} means an analytic family of compact $n$-cycles $\left\{A_t\right\}_{t\in T}$ of $X$  parametrized by $T$ such that for any $t \in T,  f\left(\left|A_t\right|\right) = \alpha(t).$}

{By the aforementioned homeomorphism and the fact that properness is a local property over the base, the following important lemma follows.}

\begin{lemma}[{e.g., \cite[Theorem 4.5]{Fu78/79}}]\label{pro}
{Let $f: X\to Y$ be a projective  morphism (Definition \ref{def-proj}). Then  each irreducible component of $\mathcal{C}_n(X/Y)$ is proper over $Y$. }
\end{lemma}

\subsection{Specialization of bimeromorphic types in    locally Moishezon smooth families}
As a warm-up for specialization of bimeromorphic types in non-smooth families, we first combine Theorem \ref{fiberbime taka mois} on fiberwise bimeromorphism with the relative Barlet cycle space theory, to prove Theorem \ref{thm-main-20250116} on the  \emph{specialization  (or deformation rigidity, or  non-deformability) of a fixed bimeromorphic structure}. 
 In fact, this proof contains all the essential ingredients for establishing the specialization of bimeromorphic types, and leads to the relationship:  
\begin{center}
fiberwise bimeromorphism $\Rightarrow$ specialization of a fixed bimeromorphic structure. 
\end{center}	

Note that  the method used in Theorem \ref{thm-main-20250116} is analytic in nature, and can be naturally carried over to the setting of locally Fujiki families (e.g., \cite{CRT26}).
            
\begin{thm}\label{thm-main-20250116}
Let $p:\mathcal{X}\to S$ be a locally Moishezon smooth family with $n$-dimensional fibers $X_t:=p^{-1}(t)$. Let $F$ be a projective manifold with $\kappa(F)\geq 0$.  Let 
$$U_b:=\{t\in S: X_t \text{ is  bimeromorphic to } F\}$$ be  the bimeromorphic locus of $F$.  Then $U_b$ is either at most a countable union of proper analytic subsets of $S$ or the entirety of $S$.	
\end{thm}
\begin{proof}
We begin by reducing the proof of this theorem to the case where $U_b$ contains a nonempty open subset of $S$  by some elementary analysis in Step \ref{step-reduction to open}. 

\begin{step}\label{step-reduction to open} 
Reduction to the case where $U_b$ contains a nonempty open subset of $S$.
\end{step}

Since the base $S$ is $1$-dimensional, a proper analytic subset of $S$ consists of discrete points and is at most countable (since the base is either projective or Stein,  it is second countable). Therefore,  to establish Theorem \ref{thm-main-20250116}, it suffices to prove: assume that $X_t:=p^{-1}(t)$ is bimeromorphic to $F$ for any $t\in B$, where $B\subseteq S$ is an uncountable subset of $S$. Then all fibers $X_t$ are bimeromorphic to $F$.		
		
We cover $S$ with at most countably many open subsets $\{S_i\}$ of $S$, each biholomorphic to the unit disk in $\mathbb{C}$, such that $p: p^{-1}(S_i)\to S_i$ is Moishezon, and that there exists a collection of relatively compact subsets $\{S_i^\prime\Subset S_i\}$ so that $\{S_i^\prime\}$ also cover $S$ and each $S_i^\prime$ is also biholomorphic to the unit disk in $\mathbb{C}$.   Since $B$ is uncountable,  we can find some $S_{i_0}$ such that $B\cap S'_{i_0}$ is also uncountable. Clearly,   there exists a projective morphism $p^{\prime}: \mathcal{X}^\prime\to S_{i_0}$ such that we have a bimeromorphic map $\tau: \mathcal{X}\dashrightarrow \mathcal{X}^\prime$ over $S_{i_0}$, based on Definition \ref{def-Moish}.

Note that, after shrinking  $S_{i_0}$ to $S'_{i_0}$,  a  desingularization $\mathcal{D}\to \mathcal{X}^{\prime}$  is the composition of finitely many blowups.  For simplicity of notation, we  denote   $S'_{i_0}$ by $S$, from now until the end of the proof of  Proposition \ref{bime-graph-many}.
Then the natural composed morphism $$\mathcal{D}\to \mathcal{X}^{\prime}\to S$$ is still projective  and the induced bimeromorphic map $\mathcal{X}\dashrightarrow \mathcal{D}$ is still over $S$. 
So from now on, we can assume that $\mathcal{X}^{\prime}$ is smooth, $p^{\prime}: \mathcal{X}^{\prime}\to S$ is projective, and $\tau: \mathcal{X}\dashrightarrow \mathcal{X}^\prime$ is a bimeromorphic map over $S$.	In particular,  $\tau|_{X_t}$ is bimeromorphic for general $t\in S$. 
Thus, for uncountably many points $t \in S$, the fibers $X_t^{\prime}$ of $p'$ are bimeromorphic to $F$.  

Now we  apply the relative Barlet cycle theory to obtain that $p^{-1}(t)$ is bimeromorphic to $F$ for general  $t\in S~ (=S'_{i_0})$. In particular, $U_b$ contains a nonempty open subset of $S$.  In fact, we can obtain more in Proposition \ref{bime-graph-many}.  The argument proceeds as follows.      
 Consider the relative $n$-cycle space  $$\mathcal{C}_n\left(\mathcal{X}^{\prime} \times F / S\right)$$
associated with the natural projective morphism $\Phi: \mathcal{X}^{\prime} \times_S (S\times F) \to S$. 
Denote by $$\varpi: \mathcal{C}_n\left(\mathcal{X}^{\prime} \times F / S\right)\to S$$ the canonical morphism.

\begin{prop}\label{bime-graph-many}
There exist  an irreducible component $$\mathcal{C}\subseteq\mathcal{C}_n\left(\mathcal{X}^{\prime} \times F / S\right)$$  
and a proper analytic subset $T$ of $S$ such that  $\mathcal{C}$ is surjective onto $S$  via $\varpi$ and all points $\lambda$ of $(\mathcal{C}\setminus (\varpi)^{-1}(T))\cap \mathcal{C}^*$   give bimeromorphic maps from the corresponding fiber $X_t^{\prime}$ to $F$,  where $t=\varpi(\lambda)$ and $\mathcal{C}^*$ is as fixed at the beginning of Subsection \ref{ss7.1}. 
In particular,  $p^{-1}(t)$ is bimeromorphic to $F$ for any general  $t\in S$.
\end{prop}

\rem Note that Timmerscheidt \cite[p. 271, line -6]{T82} stated a similar result, but no proof was given there. For comparison, certain structural aspects of the birational locus were clarified  in \cite[Proposition  2.3]{dFF13}  and \cite[Propositions 2.6 + 2.7]{BBG16}, via Hilbert scheme theoretical arguments in the scheme setting.  These results and Proposition \ref{bime-graph-many} 
 reveal different structural aspects of the  birational/bimeromorphic locus.

\begin{proof}[Proof of Proposition \ref{bime-graph-many}] \footnote{ It may be interesting to investigate whether, by following the lines of \cite[Proposition 2.3]{dFF13}  and \cite[Propositions 2.6 + 2.7]{BBG16}, one could obtain a complex-analytic analogue of \cite[Proposition 2.3]{dFF13}  and \cite[Propositions 2.6 + 2.7]{BBG16} via a purely Douady-space-theoretic approach, without invoking the Douady-to-Barlet morphism to apply the integration of cohomology classes.}
 Note that   there exist uncountably many points $t\in S$ such  that,  for each such $t$,
 there is a reduced and irreducible analytic subset $$\mathcal{G}_t^{\prime} \subseteq \Phi^{-1}(t) =X_t^{\prime} \times F,$$ which is the graph of a bimeromorphic map $X_t^{\prime} \dashrightarrow F$. We denote by $G\subseteq  \mathcal{C}\left(\mathcal{X}^{\prime} \times F / S\right)$
the set of all such  $\mathcal{G}^{\prime}_t$. 
Since $\mathcal{C}\left(\mathcal{X}^{\prime} \times F / S\right)$ has only countably many irreducible components, there exists an irreducible component $$\mathcal{C}\subseteq\mathcal{C}_n\left(\mathcal{X}^{\prime} \times F / S\right),$$      {which contains uncountably many elements of   
$G$. }
Based on Lemma \ref{pro}, $\left.\varpi\right|_\mathcal{C}: \mathcal{C}\to S$ is proper\footnote{{
Note that, by combining the homeomorphism mentioned in Section \ref{ss7.1},  the fact that properness is a local property over the base, Lemma \ref{pro} and \cite[Proposition 4.8]{Fu78/79}, one can obtain the 
following:
for a locally Moishezon family $f:X\to S$, each irreducible component of the relative $n$-cycle space $\mathcal{C}_n(X/S)$ is proper over $S$. Consequently, in the proof of Theorem \ref{thm-main-20250116}, we can also apply the theory of relative cycle spaces directly to the locally Moishezon family itself.}} . 
Since $\varpi(\mathcal{C})$ contains uncountably many $t \in S$, it follows from Remmert’s proper mapping theorem that $\varpi|_\mathcal{C}: \mathcal{C}\to S$ is surjective.   
  Moreover,  $\varpi(\mathcal{C}\cap G)$ contains uncountably many points of $S$.  

 It follows from the generic smoothness theorem   that  
 $\Phi$ and $p'$ are smooth over $S\setminus T$  for  certain proper analytic subset $T\subseteq S$.
Consequently, for any point $t\in S\setminus T$, there exists a contractible open neighborhood $D_t$ of $t$ in $S\setminus T$ such that  $\Phi$ is differentiably trivial over $D_t$.
Thus   for any fixed  $t$ and   any two points $t_1, t_2 \in D_t$,  we have isomorphisms $$H_{2n}(X^{\prime}_{t_1}\times F, \mathbb{Z}) \cong H_{2n}(\Phi^{-1}(D_t), \mathbb{Z}) \cong H_{2n}(X^{\prime}_{t_2}\times F, \mathbb{Z}).$$
Note that these isomorphisms are induced by the natural lifting and restriction  of corresponding smooth $d$-closed differential forms  representing  the corresponding  cohomology classes.

Since $\mathcal{C}$ contains uncountably many elements of $G$ and $T\subseteq S$ is at most countable, $(\mathcal{C}\setminus (\varpi)^{-1}(T))\cap \mathcal{C}^*$ is nonempty and hence   Zariski open in $\mathcal{C}$. 
Note that this irreducible complex space  $(\mathcal{C}\setminus (\varpi)^{-1}(T))\cap \mathcal{C}^*$ is arcwise connected (e.g., \cite[p. 178]{GR84}). 
Then for any $g, g^{\prime} \in (\mathcal{C}\setminus (\varpi)^{-1}(T))\cap \mathcal{C}^*$, there exist finitely many paths  $\gamma_1,\cdots,\gamma_m: [0,1]\to (\mathcal{C}\setminus (\varpi)^{-1}(T))\cap \mathcal{C}^*$ such that $\gamma_i(1)=\gamma_{i+1}(0)$ for any $i$ and that  $\gamma_1(0)=g$ and $\gamma_m(1)=g^{\prime}$ such that (the image of)
each path $\gamma_i$ is located in $\varpi^{-1}(D)\cap (\mathcal{C}\setminus (\varpi)^{-1}(T))\cap \mathcal{C}^*$ for some connected differential trivialization domain $D\subseteq S\setminus T$ of $\Phi$ ($D$ depends on $i$).
{Recall from Subsection \ref{ss7.1} that $\mathcal{C}_n(\Phi^{-1}(D)/D)$ is homeomorphic to 
$\varpi^{-1}(D)$.  Consequently, (the image of) each path  lies  in a connected component of $\mathcal{C}_n(\Phi^{-1}(D))$. }

{Let $t:=\varpi(g)$ and $s:=\varpi(g^{\prime})$.
Let
$a\in (F\times S)_t$,  $b\in (F\times S)_s$,
$c\in X'_t$ and $d\in X'_s$ be  any points, where $(F\times S)_{\bullet}$ denotes the fiber of $F\times S\to S$ over $\bullet$. 
Since the restrictions of $d$-closed smooth forms on $\Phi^{-1}(D)$ (for any fixed aforementioned $D$) to each fiber, say $X^{\prime}_t\times F$,  range  over all  cohomology classes in $H^{2n}(X^{\prime}_t\times F, \mathbb{C})$,  we can then apply  the integration of cohomology classes on compact cycles of 
$\Phi^{-1}(D)$ (\cite[Corollary 4.2.45]{BM19}) to obtain  the following equalities of intersection numbers:
\begin{align*}
[g] \cdot \left[ X_t' \times \{a\} \right] &= [g'] \cdot \left[ X_s' \times \{b\} \right], \\
[g] \cdot \left[ \{c\} \times F \right] &= [g'] \cdot \left[ \{d\} \times F \right],
\end{align*}
where $[\bullet]$ denotes the corresponding homology class associated with $\bullet$.}

Recall that for a smooth $X_t^{\prime}$, a reduced and irreducible analytic cycle 
$$\Gamma_t^{\prime} \subseteq\left(\mathcal{X}^{\prime} \times F\right)_t=X_t^{\prime} \times F,$$
corresponding to a point in $(\varpi|_\mathcal{C})^{-1}(t)$ for some $t$, 
serves as the graph of a bimeromorphic map $X_t^{\prime} \dashrightarrow F$ if and only if  $\Gamma_t^{\prime}$ has intersection number $1$ with the  fiber of $\left(\mathcal{X}^{\prime} \times F\right)_t \to X_t^{\prime}$ and it also has intersection number $1$ with the  fiber of $\left(\mathcal{X}^{\prime} \times F\right)_t \to F$ (e.g., \cite[p. 493]{GH78}).

Clearly, the intersection $G\cap ((\mathcal{C}\setminus (\varpi)^{-1}(T))\cap \mathcal{C}^*)$ is nonempty. 
Consequently,  all points $\lambda$ of $(\mathcal{C}\setminus (\varpi)^{-1}(T))\cap \mathcal{C}^*$   give bimeromorphic maps from the corresponding fiber $X_t^{\prime}$ to $F$, by  the intersection-number reason, where $t=\varpi(\lambda)$. 

Moreover, the image of  $(\mathcal{C}\setminus (\varpi)^{-1}(T))\cap \mathcal{C}^*$  under $\varpi$ is clearly a Zariski open subset of $S$.  This completes the proof of Proposition \ref{bime-graph-many}.
\end{proof}

\rem From the above argument, one readily sees that, given two families $\pi_1$ and $\pi_2$  over the same base $S$ (without any Kodaira dimension assumptions on the fibers), and under natural assumptions on the fiber product of  $\pi_1$ and $\pi_2$ (e.g., generic smoothness and  properness of the relative cycle space), the condition that  
 $\pi_1^{-1}(t)$ is bimeromorphic to  $\pi_2^{-1}(t)$ for  uncountably many $t\in S$ already implies that  $\pi_1^{-1}(t)$ is bimeromorphic to  $\pi_2^{-1}(t)$ for any general $t\in S$. 
	
Overall, we conclude from Proposition \ref{bime-graph-many} that for the proof of Theorem \ref{thm-main-20250116}, it suffices to prove:
\begin{prop}\label{specialization-base 1 dim-bimero}
Let $p:\mathcal{X}\to S$ be a locally Moishezon smooth family with $n$-dimensional  fibers $X_t:=p^{-1}(t)$. 
Let $F$ be a projective manifold with $\kappa(F)\geq 0$.  
Assume that $U_b$ contains a nonempty open subset of $S$, where $U_b$ is as in the statement of Theorem \ref{thm-main-20250116}. Then   $X_t$ is bimeromorphic to $F$ for any $t\in S$.   
\end{prop}

In the remaining part of the proof of Theorem \ref{thm-main-20250116},  Steps \ref{s2}, \ref{step2-Barlet} and \ref{step3-apply fiberwise bime} are devoted to proving Proposition \ref{specialization-base 1 dim-bimero}.
	
\begin{step}\label{s2}
Reduction of Proposition \ref{specialization-base 1 dim-bimero} to  the verification of  bimeromorphism for fibers near certain fixed fiber. 
\end{step}		

For the reader’s convenience and  to  clarify the subsequent steps, we include the full details of an easy topological argument.
 Denote by $V$ the interior of the bimeromorphic locus $U_b$. 
Then $V$ is nonempty. As $S$ is connected, to prove  Proposition \ref{specialization-base 1 dim-bimero}, it suffices to show that $V$ is closed.

Let  $\{U_{\alpha}\}_{\alpha}$ be an open cover of $S$ with each $U_{\alpha}$ being biholomorphic to  the unit disk $\Delta\subseteq \mathbb{C}$.   It suffices to prove that 
$$\overline{V}\cap U_{\alpha} \subseteq {V}\cap U_{\alpha}$$
for each $\alpha$ with $\overline{V}\cap U_{\alpha}\neq \emptyset$.  Fix such a $U_{\alpha}$ and a point $0\in \overline{V}\cap U_{\alpha}$, if $0\in V$, there is nothing to prove. 
So it suffices to prove $0\in V$ when $0\in \partial V \cap U_{\alpha}$, i.e., we should establish the existence of an open neighborhood $W_0$ of $0$ such that for any $t\in W_0$, $X_t$ is bimeromorphic to $F$.

From the above analysis,  in the setting of Proposition \ref{specialization-base 1 dim-bimero}, we can assume that the base space $S$ is the unit disc $\Delta\subseteq \mathbb{C}$, $p$ is Moishezon, and there exists a nonempty open subset $V$ of $S$ such that $X_t$ is bimeromorphic to $F$ for any $t\in V$. 
Additionally, $0\in \partial V$.  

To complete the proof of Proposition \ref{specialization-base 1 dim-bimero}, and hence the proof of  Theorem \ref{thm-main-20250116}, we should establish the existence of an open neighborhood $W_0$ of $0$ such that for any $t\in W_0$, $X_t$ is bimeromorphic to $F$.  
In the next two steps, we will utilize the relative Barlet cycle space theory and our fiberwise bimeromorphism result to reach this.

\begin{step}\label{step2-Barlet}
Construction of a global bimeromorphic map between the total spaces  after base change.
\end{step}

Similarly to what was done in Step \ref{step-reduction to open}, by the Moishezonness of $p$, we can obtain a projective morphism $$p^{\prime}: \mathcal{X}^{\prime}\to S$$ from a connected complex manifold  
$\mathcal{X}^{\prime}$ such that there exists a bimeromorphic map $\tau: \mathcal{X}\dashrightarrow \mathcal{X}^\prime$  over $S$.	  In particular,  we can assume that $\tau|_{X_t}$ is bimeromorphic for $t\neq 0$ (after shrinking the base $S$). Consequently, for any $t\in V\setminus \{0\}$, $X_t^{\prime}$ is bimeromorphic to $F$, where $X_t^{\prime}$ is the  fiber  of  $p^{\prime}$ over $t$.  	Moreover,  $p^{\prime}$ is a flat morphism with irreducible general fibers, based on  the criterion for flatness of a morphism over a smooth curve. 

It is worth noting that, as in Step \ref{step-reduction to open}, the generic smoothness theorem implies the existence of a proper analytic subset $T$ of $S$ such that both   the natural morphism $\mathcal{X}^{\prime} \times_S (S\times F) \to S$ and $p'$ are smooth projective morphisms over $S\setminus T$.

Now we apply a relative Barlet cycle space theoretical argument,  motivated by  Step $e$ of the proof of \cite[Theorem 1]{T82} and by the proof of \cite[Theorem 3.1]{dFF13}, to construct   a bimeromorphic map 
$$\nu: \widetilde{\mathcal X}^{\prime} \dashrightarrow \tilde{\mathcal{C}}\times F,$$ where $\widetilde{\bullet}$ is the base extension of $\bullet$ under some morphism $\tilde{\mathcal{C}}\to S$.

As in Proposition \ref{bime-graph-many} of Step \ref{step-reduction to open}, there exists  an irreducible component $\mathcal{C}$ of the relative $n$-cycle space
    $\mathcal{C}_n\left(\mathcal{X}^{\prime} \times F / S\right)$ of the morphism $\mathcal{X}^{\prime} \times_S (S\times F) \to S$   
 such that  $\mathcal{C}$ is surjective onto $S$  via $\varpi$ and all points $\lambda$ of $(\mathcal{C}\setminus (\varpi)^{-1}(T))\cap \mathcal{C}^*$ (which is Zariski open in $\mathcal{C}$)  give bimeromorphic maps from the corresponding fiber $X_t^{\prime}$ to $F$,  where  $t=\varpi(\lambda)$ .

We now take an irreducible  local curve $\mathcal{C}^{\prime}$ of
$\mathcal{C}$ in an open subset $U\subseteq \mathcal{C}$  of $\mathcal{C}$ such that $\mathcal{C}^{\prime}$ intersects both $(\varpi|_\mathcal{C})^{-1}(0)$ and $(\mathcal{C} \cap \mathcal{C}^*\setminus (\varpi|_\mathcal{C})^{-1}(T))\setminus (\varpi|_\mathcal{C})^{-1}(0)$.  In particular, $\varpi|_\mathcal{C}(\mathcal{C}^{'})$ contains an open neighborhood of $0$, and  the general point $\lambda'$ of $\mathcal{C}^{\prime}$  gives a bimeromorphic map from the corresponding fiber $X_t^{\prime}$ ($t=\varpi(\lambda')$) to $F$.

 Let $\eta: \tilde{\mathcal{C}} \to  \mathcal{C}^{\prime}$ be the normalization of $\mathcal{C}^{\prime}$, and denote by $\tilde{\nu}$ the natural morphism $ \tilde{\mathcal{C}} \to  S$. 
Let
$$
\widetilde{\mathcal X^{\prime} \times F}:=\tilde{\mathcal{C}} {\times}_S\left(\mathcal X^{\prime} \times F\right), \quad \widetilde{\mathcal X^{\prime}}:={\tilde{\mathcal{C}}} {\times}_S {\mathcal X}^{\prime}, \quad \widetilde{\mathcal X}:={\tilde{\mathcal{C}}}{\times}_S \mathcal X.
$$

Note that we have a natural composed morphism
$$\Lambda: \tilde{\mathcal{C}}\to \mathcal{C}^{\prime}\to U 
 \to \mathcal{C}\to \mathcal{C}_n\left({\mathcal X^{\prime} \times F} / {S}\right) \to \mathcal{C}_n\left(\mathcal X^{\prime} \times F \right),$$ where  $U 
 \to \mathcal{C}$ is the natural inclusion and $\mathcal{C}^{\prime}\to U$ is the natural closed embedding.  
It then follows from Lemma \ref{analytic subset},  that 	
$$G_1:=\left\{(\tilde{c}, P) \in \tilde{\mathcal{C}} \times{\left(\mathcal X^{\prime} \times F\right)} \mid P \in |\Lambda(\tilde{c})| \right\}
$$
is an analytic subset of $\tilde{\mathcal{C}} \times (\mathcal X^{\prime} \times F)$ that is proper over $\tilde{\mathcal{C}}$.  

Since $\Lambda(\tilde{c})$ is an $n$-cycle for any $\tilde c \in \tilde{\mathcal{C}}$, each fiber of the induced morphism $G_1\to \tilde{\mathcal{C}}$  is of pure dimension, where $G_1$ is equipped with the reduced structure. Moreover, each irreducible component of  $G_1$ is of dimension $n+1$ (e.g., \cite[Section 4.3.3]{BM19}).  Hence  each irreducible component of $G_1$ is onto $\tilde{\mathcal{C}}$.
Consequently, the morphism
$G_1\to \tilde{\mathcal{C}}$ is flat (e.g., \cite[Introduction]{Hi75}), and thus $G_1$  is irreducible, based on the fact that  $|\Lambda (\tilde c)|$ is irreducible for general $\tilde c\in \tilde{\mathcal{C}}$.   Moreover, for general $\tilde c\in \tilde{\mathcal{C}}$, 
$$G_1 \cap \{\tilde c\}\times (\mathcal X^{\prime} \times F) \subseteq  \{\tilde c\}\times  ((\mathcal X^{\prime})_{\tilde{\nu} (\tilde c)} \times (F\times S)_{\tilde{\nu} (\tilde c)})$$
serves as the graph of a bimeromorphic map
$(\mathcal X^{\prime})_{\tilde{\nu} (\tilde c)}\dashrightarrow F\cong (F\times S)_{\tilde{\nu} (\tilde c)}$.

Consider the natural closed embedding
$$
j:=(q, \tilde{\nu}^{\prime}): \widetilde{\mathcal{X}^{\prime}\times F} \to \tilde{\mathcal{C}} \times (\mathcal{X}^{\prime} \times F), \quad z \mapsto(q(z), \tilde{\nu}^{\prime}(z)),
$$
where $\tilde{\nu} '$ denotes  the natural morphism 
$\widetilde{\mathcal X^{\prime} \times F}\to \mathcal X^{\prime} \times F$, and $q$ denotes the natural morphism $\widetilde{\mathcal{X}^{\prime}\times F} \to \tilde{\mathcal{C}}$.
Clearly,    $G_1$ is located in the image of $j$. It then follows that  the analytic subset
$G_2:=j^{-1}(G_1)$ of $$\widetilde{\mathcal{X}^{\prime}\times F}\cong \widetilde{\mathcal{X}^{\prime}}\times_{\tilde{\mathcal{C}}} (\tilde{\mathcal{C}}\times F)$$
is also irreducible as $G_1$ is.

Note that $\tilde{\nu} '$ is fiberwise biholomorphic (e.g., \cite[p. 29]{Fs76}), i.e., it induces the isomorphism
$$(\widetilde{\mathcal X^{\prime} \times F})_{\tilde c}  \cong   (\mathcal X^{\prime} \times F)_{\tilde{\nu} (\tilde c)}$$
for each $\tilde c\in \tilde{\mathcal{C}}$.
Then 
for general $\tilde c\in \tilde{\mathcal{C}}$, 
$$G_2 \cap q^{-1}(\tilde{c})=
(\tilde{\nu} '|_{(\widetilde{\mathcal{X}^{\prime} \times F})_{\tilde{c}}})^{-1}(\rm{pr}_{2, \tilde{c}} (G_1 \cap\{\tilde{c}\} \times(\mathcal{X}^{\prime} \times F)))
$$ 
  serves as the graph of a bimeromorphic map
$(\widetilde{\mathcal X^{\prime} })_{\tilde c}\dashrightarrow F$, where $\rm{pr}_{2, \tilde{c}}$ denotes the natural projection $\{\tilde{c}\} \times(\mathcal{X}^{\prime} \times F) \to \mathcal{X}^{\prime} \times F$.

Since flatness is stable under base change and  the fiber is preserved by base change (e.g., \cite[p. 29]{Fs76}),  the morphism $\tilde{p'}: \widetilde{\mathcal{X}'}\to \tilde{\mathcal{C}}$ is  flat and its general fiber is irreducible. Consequently,  
$\widetilde{\mathcal{X}^{\prime}}$  is  irreducible.  Thus,  	
$G_2$ serves as the graph of a bimeromorphic map $$\nu: \widetilde{\mathcal X^{\prime}} \dashrightarrow  \tilde{\mathcal{C}} \times F$$ over $\tilde{\mathcal{C}}$.

\begin{step}\label{step3-apply fiberwise bime}
Application of fiberwise bimeromorphism result.
\end{step}

Following the reasoning in the preceding paragraph, it follows from  Lemmata \ref{jiayibingding} and \ref{connectedness of total space} that 
$\tilde{p}: \widetilde{\mathcal{X}}\to \tilde{\mathcal{C}}$  is smooth and that $\widetilde{\mathcal{X}}$ is a connected manifold.
So $\widetilde{\mathcal{X}}$  is  bimeromorphic to $\widetilde{\mathcal{X}^{\prime}}$	over $\tilde{\mathcal{C}}$, since $\mathcal{X}$ is bimeromorphic to $\mathcal{X}^{\prime}$ over $S$.
			
Consequently, one obtains that $\tilde{\mathcal{C}}\times F$ is bimeromorphic to  $\widetilde{\mathcal{X}}$ over $\tilde{\mathcal{C}}$, by composing with the $\nu$ constructed in Step \ref{step2-Barlet}.  Note that $\widetilde{\mathcal{X}}\to \tilde{\mathcal{C}}$ is  locally Moishezon, which follows from   the  stability of projectivity  under  base change. 	As a result, by Theorem \ref{fiberbime taka mois} or Corollary \ref{fibe bime canonical singu-kodaira geq 0}, we conclude that $\tilde{\mathcal{C}}\times F$ is fiberwise bimeromorphic to  $\widetilde{\mathcal{X}}$ over $\tilde{\mathcal{C}}$. This is because, for the smooth family $\tilde p$, the upper semi-continuity of plurigenera is ensured by Kodaira--Spencer's upper semi-continuity theorem.  Again utilizing the fact that base change does not change the fiber  (e.g., \cite[p. 29]{Fs76}), we obtain that each fiber near $X_0$ is bimeromorphic to $F$, as desired in the last paragraph of Step \ref{s2}. Thus, we complete the proof of Proposition \ref{specialization-base 1 dim-bimero}, and so that of Theorem \ref{thm-main-20250116}.
\end{proof}

\rem{Based on Proposition \ref{specialization-base 1 dim-bimero}, one can easily derive some related results as the base $S$ is of arbitrary dimension.  For example: Let $p:\mathcal{X}\to B$ be a locally Moishezon smooth morphism, where $B$ is a complex manifold of dimension $\geq 1$. Assume that each fiber over $B\setminus \Gamma$ is bimeromorphic to $F$ with $\Gamma$ a proper analytic subset of $B$,  where $F$ is a projective manifold with $\kappa(F)\geq 0$, then all the fibers are bimeromorphic to $F$.  In fact, for any point $s\in \Gamma$, we can find a  $1$-dimensional smooth analytic disk $\Delta$ centered at $s$ and $\Delta\cap \Gamma$ is a proper analytic subset of $\Delta$. Then this result follows easily from  \cite[Theorem 1.4.(ii)]{RT21} or \cite[Corollary 22]{Kol22} and Proposition \ref{specialization-base 1 dim-bimero}.}
		
As a direct application of Theorem \ref{thm-main-20250116},  we can easily get Example \ref{P1 local rigidity-new}, based on the well-known criteria for local rigidity of a compact complex manifold in the local deformation theory. This also provides an instance of the condition ``$X_t$ is bimeromorphic to $F$ for all $t$ in some nonempty open subset of $S$" in Proposition \ref{specialization-base 1 dim-bimero}.
\begin{ex}\label{P1 local rigidity-new}
Let $\pi:\mathcal{X}\to S$ be a locally Moishezon smooth family with compact complex manifolds $X_t:=\pi^{-1}(t)$.
Assume that for some $s\in S$, $X_s$ satisfies $\kappa({X_s})\geq 0$ and $H^1(X_s, \Theta_{X_s})=0$, where $\Theta_{X_t}$ is the tangent sheaf of $X_t$. Then $X_t$ is bimeromorphic to $X_s$ for any $t\in S$. 
\end{ex}

    {Motivated by Mumford–Persson \cite[Lemma 3.1.1, Proposition
3.1.2]{Ps77}\footnote{Below \cite[Proposition 2.6.1]{Ps77}, in page 61, Persson wrote, ``This observation (pointed out to me by Mumford)".}, }
we can easily modify the proof of Theorem \ref{thm-main-20250116} to obtain Theorem \ref{specialization-base 1 dim-bimero-no global bime assum-uncountable}, which  may  be  viewed as  the  converse of Catanese's birationality openness result.  
Recall that \cite[Theorem 2.4]{Ct23} states a condition ensuring that having a birational map onto the image is an open property for a certain family whose ``central" fiber is an irreducible normal non uniruled variety.  

Compared to Theorem \ref{fiberbime taka mois} or Corollary \ref{fibe bime canonical singu-kodaira geq 0} (when both $\pi_1$ and $\pi_2$ are smooth), 
Theorem \ref{specialization-base 1 dim-bimero-no global bime assum-uncountable} does not require the total spaces of the two families involved to be bimeromorphic over the base, a priori. This type of result is referred to as the \emph{specialization of bimeromorphic types} in this paper, motivated by the notion of specialization of birational types in \cite{KT19}.   
Note that both \cite{KT19} and \cite{NO21} dealt with schemes while the fibers of Theorem \ref{specialization-base 1 dim-bimero-no global bime assum-uncountable} are Moishezon, which correspond to complete algebraic spaces. 	
\begin{thm}\label{specialization-base 1 dim-bimero-no global bime assum-uncountable}
Let  $p:\mathcal{X}\to S$ and  $q:\mathcal{Y}\to S$ be two locally Moishezon smooth families and each fiber of $\pi_1$ of Kodaira dimension not equal to $-\infty$.  
Set $$U_b:=\{t\in S: X_t \text{ is bimeromorphic to } Y_t\}.$$
Then $U_b$ is either an at most countable union of proper analytic subsets of $S$ or the whole of $S$.
\end{thm}
\begin{proof}
The proof of this theorem is almost identical to that  of Theorem \ref{thm-main-20250116}, and so we only indicate the key points.  As in Theorem \ref{thm-main-20250116}, it suffices to prove the following: if $U_b$ is uncountable, then $U_b=S$.  We also divide the proof into four steps, as follows.
\begin{enumerate}[\rm{(}1\rm{)}]
\item
Take a connected open subset  $U\subseteq S$ such that 
  $p$ is  bimeromorphic over $U$ to a projective morphism $p^{\prime}: \mathcal{X}^{\prime}\to U$ with $\mathcal{X}^{\prime}$ smooth, and $q$  is  also bimeromorphic over $U$ to a projective morphism $q^{\prime}: \mathcal{Y}^{\prime}\to U$ with $\mathcal{Y}^{\prime}$ smooth. Moreover,   $(p^{\prime})^{-1}(t)$ is bimeromorphic to $(q^{\prime})^{-1}(t)$ for uncountably many  $t\in U$.     Since projectivity of morphism is stable under base change and the composition of two projective morphisms is also projective, the composition  $$\alpha:\mathcal{X}^{\prime} \times_U \mathcal{Y}^{\prime} \to \mathcal{Y}^{\prime} \to U$$ is still projective. 

  Note that there exists a proper analytic subset $T\subseteq U$ such that $p',  q'$ and $\alpha$ are smooth projective morphisms over $U\setminus T$.
Consequently, as in    Step \ref{step-reduction to open} of  Theorem  \ref{thm-main-20250116}, via  analyzing the cycle space   $\mathcal{C}\left(\mathcal{X}^{\prime} \times_U \mathcal{Y}^{\prime} / U\right)$,    we obtain that  $(p^{\prime})^{-1}(t)$ is bimeromorphic to $(q^{\prime})^{-1}(t)$ for general $t\in U$.  

\item\label{temp-2}  
As  in Step  \ref{s2} of Theorem \ref{thm-main-20250116}, we reduce the question to the local case: the base space $S$ (or say $U$) is the unit disc $\Delta\subseteq \mathbb{C}$, $p$ and $q$ are Moishezon morphisms over $S$,  and there exists a nonempty open subset $V$ of $S$ such that $X_t$ is bimeromorphic to $Y_t$ for any $t\in V$, and $0\in \partial V$. For the proof of this theorem,  it suffices to establish the existence of an open neighborhood $W_0$ of $0$ such that for any $t\in W_0$, $X_t$ is bimeromorphic to $Y_t$.

\item
As in Step \ref{step2-Barlet}   of  Theorem \ref{thm-main-20250116}, since $p$ and $q$ are Moishezon morphisms, 
 $p$ is  bimeromorphic over $S$ to a projective morphism $p^{\prime}: \mathcal{X}^{\prime}\to S$ with $\mathcal{X}^{\prime}$ smooth, and $q$  is  also bimeromorphic over $S$ to a projective morphism $q^{\prime}: \mathcal{Y}^{\prime}\to S$ with $\mathcal{Y}^{\prime}$ smooth. 
 Using  the 
relative  cycle space 	$\mathcal{C}\left(\mathcal{X}^{\prime} \times_S \mathcal{Y}^{\prime} / S\right)$,
we obtain a bimeromorphic map  $\widetilde{\mathcal{X}^{\prime}}\dashrightarrow \widetilde{\mathcal{Y}^{\prime}}$   over $\tilde{\mathcal{C}}$, where $\tilde{\mathcal{C}}$ is a smooth connected curve and  $\widetilde{\bullet}$ is the base change of $\bullet$ under some morphism $\tilde{\mathcal{C}}\to S$ (here $S$ is an open neighborhood of $0$ in the $S$ appearing in item \eqref{temp-2}).

\item 
As in Step \ref{step3-apply fiberwise bime}   of  Theorem \ref{thm-main-20250116}, we obtain a bimeromorphic map  $\widetilde{\mathcal{X}}\dashrightarrow \widetilde{\mathcal{Y}}$   over $\tilde{\mathcal{C}}$.
 Since projectivity is stable  under base change, both  $\widetilde{\mathcal{X}}\to \tilde{\mathcal{C}}$ and $\widetilde{\mathcal{Y}}\to \tilde{\mathcal{C}}$ are locally Moishezon.  Applying Corollary \ref{fibe bime canonical singu-kodaira geq 0} to the bimeromorphic map $\widetilde{\mathcal{X}}\dashrightarrow \widetilde{\mathcal{Y}}$ over $\tilde{\mathcal{C}}$ then yields that $(\widetilde{\mathcal{X}})_{\tilde{c}}$ is bimeromorphic to $(\widetilde{\mathcal{Y}})_{\tilde{c}}$ for any $\tilde{c}\in \tilde{\mathcal{C}}$. Consequently,  $X_t$ is bimeromorphic to $Y_t$, as desired in item \eqref{temp-2}, based on the fact that base change does not change the fiber (e.g., \cite[p. 29]{Fs76}). So we complete the proof.
\end{enumerate}
\end{proof}	
	
Note that the result on the specialization of bimeromorphic types seems much stronger than that on the specialization of a fixed bimeromorphic structure. For example, Theorem \ref{thm-main-20250116}  is an easy corollary of Theorem \ref{specialization-base 1 dim-bimero-no global bime assum-uncountable}.
However, comparing the proof of Theorem \ref{specialization-base 1 dim-bimero-no global bime assum-uncountable} with that of Theorem \ref{thm-main-20250116}, we can deal with the non-deformability of a bimeromorphic structure and the specialization of bimeromorphic types in a unified way, by virtue of the characterization of fiberwise bimeromorphism.  That is to say,  we obtain the relation:
	
specialization of a fixed bimeromorphic structure
	$\underset{\scalebox{0.5}{\(\begin{array}{c}\text{method}\  \text{level}\end{array}\)}}{\Longleftrightarrow}$ specialization  of bimeromorphic types.

\subsection{Specialization of   bimeromorphic types in non-smooth Moishezon families}	
	
In this subsection,  we apply the method which is almost identical to the one presented in the proof of Theorem \ref{thm-main-20250116} or \ref{specialization-base 1 dim-bimero-no global bime assum-uncountable} to obtain Theorem \ref{specialization-base 1 dim-singular family} for certain locally Moishezon non-smooth families.

\begin{defn}\label{def-universally normal}
Let $p:\mathcal{X}\to S$ be a family. The family $p$ is called \emph{universally normal} if the base change   $\mathcal{X}\times_S \tilde{S}$  of the total space  $\mathcal{X}$  is normal  for any morphism $\tilde{S}\to S$ with $\tilde{S}$ also a connected smooth curve.
\end{defn}
	
\begin{ex}\label{univnormal-cano}
 Clearly, a smooth family is universally normal. More generally, a family in which each fiber has only canonical singularities is also universally normal, due to the   deformation behavior of canonical singularities by Kawamata \cite{Kw99a} or Fujino \cite{Fj23}. In fact,  they established  that if a reference fiber  in a flat family has only canonical singularities, then the total space $\mathcal{Z}$ near this reference fiber has only canonical singularities, in particular, $\mathcal{Z}$ is normal.  
\end{ex}

Since Lemma \ref{pro} makes no singularity assumption on the morphism involved,  we can then	apply the method almost identical  to  the one presented in the proof of  Theorem \ref{thm-main-20250116}  or \ref{specialization-base 1 dim-bimero-no global bime assum-uncountable} to obtain the following result on specialization of bimeromorphic types for universally normal locally Moishezon families.  Note that the condition ``universally normal'' is included to ensure the normality of the base-changed total space, which is required for applying Theorem  \ref{fiberbime taka mois}  or Corollary \ref{fibe bime canonical singu-kodaira geq 0}.  
	
\begin{thm}\label{specialization-base 1 dim-singular family}		
Let  $\pi_1:\mathcal{X}\to S$ and  $\pi_2:\mathcal{Y}\to S$ be two universally normal,  locally Moishezon families, 
where each fiber of $\pi_1$ is  irreducible and   of non-negative Kodaira dimension, and each fiber of $\pi_2$ is also irreducible. 
Moreover, for any  $s\in S$, assume $P_m(Y_s)\geq P_m(Y_t)$ for any $t$ near $s$. 
Set
$$U_b:= \{t\in S: \text{ the reduction of } X_t \text{ is bimeromorphic to  the reduction of } Y_t \}.$$
Then $U_b$ is either an at most countable union of proper analytic subsets of $S$ or the whole of $S$. 
\end{thm}

Example \ref{univnormal-cano} and the proof of \cite[Theorem 1.1]{Tk07} together imply that a family in which each fiber has only canonical singularities can  serve as an example of the family $\pi_2$ in Theorem \ref{specialization-base 1 dim-singular family}, as $\pi_2$ is required to satisfy the upper semi-continuity  and  universal-normality conditions. This motivates an observation in: 

\rem 
Rao--Tsai \cite[Proof of Theorem 1.2]{RT22} and Koll\'ar \cite[Corollary 22]{Kol22} proved that a smooth family with each fiber being Moishezon is  locally Moishezon. Li--Rao--Wang--Wang \cite[Theorem 1.1]{LRWW24} proved that each fiber of a smooth family with uncountably many fibers being Moishezon is  Moishezon.  Motivated by those results and  Koll\'ar's conjectures and examples \cite[Conjectures 1, 3 + Example 14]{Kol22} on Moishezon morphisms, we conjecture that a  family (automatically flat over a smooth curve in this paper) in which each fiber has only canonical singularities  is locally Moishezon  if and only if uncountably many fibers are Moishezon.  If this conjecture is answered affirmatively,  one can then easily obtain a result, which is a corollary to Theorem \ref{specialization-base 1 dim-singular family}, without the local Moishezonness assumption as imposed there:
Let $\pi_1: X \to S$ and $\pi_2: Y \to S$ be two families in which each  fiber  has only canonical singularities. Assume that each fiber of $\pi_1$ is of Kodaira dimension not equal to $-\infty$. 
Set
$$
U_b:=\{t \in S: X_t \text{ is  Moishezon  and is bimeromorphic to } Y_t \} .
$$
Then $U_b$ is either an at most countable union of proper analytic subsets of $S$ or the whole of $S$.

\begin{defn}\label{def-strongly bime}
Let  $\pi: \mathcal{X}\to S$   be a family. The family $\pi$ is called \emph{locally strongly bimeromorphically isotrivial}, if for any reference point $s\in S$, the base change $\mathcal{X}_{S^{\prime}}:= \mathcal{X}\times_S  S^{\prime}$ is fiberwise bimeromorphic to a smooth trivial family $F\times S^{\prime} \to S^{\prime}$ over $S^{\prime}$, where $F$ is some compact complex manifold and $S^{\prime}\to S$ is some surjective morphism between two smooth curves, possibly after shrinking $S$ around $s$. 
\end{defn}
	
Motivated by the result of Bogomolov--B\"ohning--Graf von Bothmer \cite[Corollary 1.3]{BBG16}, we can easily apply the argument of Theorem \ref{thm-main-20250116} or \ref{specialization-base 1 dim-bimero-no global bime assum-uncountable} to obtain the following result on locally strongly bimeromorphically isotriviality. Recall that \cite[Corollary 1.3]{BBG16} proved: Let $g: U \rightarrow B$ be a family of algebraic varieties (in the scheme setting) such that all fibers
are birational to each other and B is integral. Then $g$ is birationally isotrivial. 

\begin{cor}\label{cor-strongly bime}
Let $\pi:\mathcal{X}\to S$  be a  universally normal locally Moishezon families, where each fiber of $\pi$ is an irreducible complex analytic space of Kodaira dimension not equal to $-\infty$ and uncountably many fibers of $\pi$ are bimeromorphic to one fixed projective manifold. Then   $\pi:\mathcal{X}\to S$  is locally strongly bimeromorphically isotrivial.
\end{cor}
		
Furthermore, we present the following observation in a very rough manner, which appears somewhat artificial in nature. However, this observation indicates some clues that the singularity (see Observation \ref{specialization-base 1 dim-singular family-0726}.\eqref{submersion point 0726}) of the corresponding relative Barlet cycle space of some space (which involves information of the two families we are studying) over the base can influence the behavior of the specialization of bimeromorphic types.
	
\begin{observation}\label{specialization-base 1 dim-singular family-0726}
Let  $\pi_1:\mathcal{X}\to S$ and  $\pi_2:\mathcal{Y}\to S$ be two locally projective families (not necessarily universally normal), where each fiber of $\pi_1$ is an irreducible complex analytic space of Kodaira dimension not equal to $-\infty$ and each fiber of $\pi_2$ is also irreducible. Moreover, for any $t$ in a nonempty analytic Zariski open subset $\Lambda$ of $S$, $X_t:=(\pi_1)^{-1}(t)$  is bimeromorphic to $Y_t:=(\pi_2)^{-1}(t)$. 
Moreover, for any  $s\in S$,  suppose $P_m(Y_s)\geq P_m(Y_t)$ for any $t$ near $s$. 
Additionally, we assume that there exists an irreducible component $\mathcal{C}$ of the relative Barlet cycle space $\mathcal{C}\left(\mathcal{X} \times_S \mathcal{Y} / S\right)$ that satisfies the property:		
\begin{enumerate}[\rm{(}1\rm{)}]
			\item
			$\mathcal{C}$ contains uncountably many points whose image (under the natural morphism $\mathcal{C}\to S$ ) in $S$ corresponds to uncountably many $t$;
			\item
			Each of those above uncountably many points  serves as the graph of the bimeromorphic map $X_t \dashrightarrow Y_t$ for the corresponding $t$;
			\item \label{submersion point 0726}
			Each  fiber over $S\setminus \Lambda$ of the natural morphism $\mathcal{C}\to S$ has a smooth  point (this condition may not be superfluous, as shown in 
\cite{Ct89}). 
\end{enumerate}
Then $X_t$ is bimeromorphic to $Y_t$ for any $t\in S$.
\end{observation}
\begin{proof}
This proof is almost identical to that in Theorem \ref{thm-main-20250116}  or \ref{specialization-base 1 dim-bimero-no global bime assum-uncountable}, except that the local curve $\mathcal{C}^{\prime}$ in Step \ref{step2-Barlet} of the proof of Theorem \ref{thm-main-20250116}  can be chosen to be a local section of $\mathcal{C}\to S$ that is biholomorphic to $S$ (after shrinking the base $S$). In fact, the morphism $\mathcal{C}\to S$ is flat by the flatness criterion for a proper morphism over a smooth curve. Consequently,   $\mathcal{C}\to S$  is submersive    at some point of the fiber over $S\setminus \Lambda$  of $\mathcal{C}\to S$ by the condition \eqref{submersion point 0726} and Lemma \ref{jiayibingding}. 
		
Consequently, we can choose the local curve $\mathcal{C}^{\prime}$ in Step \ref{step2-Barlet} of the proof of Theorem \ref{thm-main-20250116} to be biholomorphic to $S$ (after shrinking $S$). Then the morphism $\tilde{\mathcal{C}}\to S$ in Step \ref{step2-Barlet} of the proof of Theorem \ref{thm-main-20250116} is biholomorphic and thus the corresponding base changes induced by $\tilde{\mathcal{C}}\to S$ do nothing.  Then the bimeromorphic map $\widetilde{\mathcal{X}} \dashrightarrow\widetilde{\mathcal{Y}}$ over $\tilde{\mathcal{C}}$ is exactly the desired bimeromorphic map ${\mathcal{X}} \dashrightarrow {\mathcal{Y}}$ over ${S}$. As a result, Theorem \ref{fiberbime taka mois} or Corollary \ref{fibe bime canonical singu-kodaira geq 0} implies that $\pi_1$ is fiberwise bimeromorphic to $\pi_2$. This completes the proof of this observation. 
\end{proof}
		
\subsection{On Koll\'ar's conjecture on fiberwise bimeromorphic model}
Recently, on the existence of the fiberwise bimeromorphic model for certain Moishezon families, Koll\'ar \cite{Kol22} proposed:
	
\begin{con}[{\cite[Conjecture 5]{Kol22}}]\label{conj5}
Let $g: \mathcal X \to \Delta$ be a flat, proper, Moishezon morphism over the unit disc $\Delta$ in $\mathbb{C}$. Assume that $X_0$ has only canonical singularities (Definition \ref{def-cano singu}). Then $g$ is fiberwise bimeromorphic  to a flat, projective morphism $g^{\mathrm{p}}:\mathcal X^{\mathrm{p}} \to \Delta$ (possibly over a smaller disc) such that
\begin{enumerate}[\rm{(}1\rm{)}]
		\item
			$X_0^{\mathrm{p}}:=(g^{\mathrm{p}})^{-1}(0)$ has only canonical  singularities,
			\item
			$X_s^{\mathrm{p}}:=(g^{\mathrm{p}})^{-1}(s)$ has terminal singularities for $s \neq 0$, and
			\item
			$K_{\mathcal X^{\mathrm{P}}}$ is $\mathbb{Q}$-Cartier.
		\end{enumerate}
\end{con}

Koll\'ar \cite[Theorem 28]{Kol22} gives a positive answer to the log terminal case of his conjecture, provided $X_0$ is not uniruled. 
Furthermore, Koll\'ar \cite[p. 1674, 32]{Kol22} also gives a proposal for addressing the canonical case.     
Furthermore,  Koll\'ar also constructs a non-projective Moishezon morphism and gives a conjecture (\cite[Conjecture 27.1]{Kol22}) that in most cases,  this morphism is not fiberwise bimeromorphic to a flat, projective morphism. 
	
We expect that our result on fiberwise bimeromorphism may help to study topics related to Koll\'ar's conjecture through the following  approach. The solution to Koll\'ar's conjecture may be split into  two steps:
\begin{enumerate}[\rm{(}1\rm{)}]
		\item
		$g:\mathcal X \to \Delta$ (possibly over a smaller disc)  is bimeromorphic via $f$ to a flat projective morphism   $g^{\mathrm{p}}:\mathcal X^{\mathrm{p}} \to \Delta$ (possibly over a smaller disc)  with the desired properties in Koll\'ar's conjecture;
		\item
		The bimeromorphic map $f$ is  a fiberwise bimeromorphism based on Theorem  \ref{fiberbime taka mois}.
		% \ref{fiberbime taka mois-without kod dim assump}. 
\end{enumerate}
Our fiberwise bimeromorphism results may then be used in the second step in addressing Koll\'ar's conjecture.  
	
Based on the spirit of our arguments on specialization questions  presented in this section (e.g., Theorem \ref{specialization-base 1 dim-bimero}), we propose: 
\begin{con}
Let $g: \mathcal X \to \Delta$ be a flat, proper, Moishezon morphism, where $\Delta$ is the unit disc in $\mathbb{C}$. Assume that $X_0$ has only canonical singularities (Definition \ref{def-cano singu}). Then after shrinking the $\Delta$,  there exists some surjective morphism $\nu:\tilde{\Delta}\to \Delta$ with $\tau\in \nu^{-1}(0)$ such that the base change $\tilde{g}$ of $g$ induced by $\nu$ is fiberwise bimeromorphic  to (possibly after shrinking $\tilde{\Delta}$) a flat, projective morphism $g^{\mathrm{p}}:\mathcal X^{\mathrm{p}} \to \tilde{\Delta}$ such that
		\begin{enumerate}[\rm{(}1\rm{)}]
			\item
			$(g^{\mathrm{p}})^{-1}(\tau)$ has only canonical  singularities,
			\item
			$(g^{\mathrm{p}})^{-1}(s)$ has terminal singularities for $s \neq \tau$, and
			\item
			$K_{\mathcal X^{\mathrm{P}}}$ is $\mathbb{Q}$-Cartier.
		\end{enumerate}
	\end{con}

 \section*{Acknowledgment}
{The authors would like to thank Professor Hsueh-Yung Lin for kindly providing Examples \ref{Hirze example} and \ref{exc not uniruled}, and for his helpful explanations of these examples.
We are grateful to Professor J. Koll\'ar for pointing out Corollary \ref{Kol-abso}.
We also thank Dr. Yi Li for bringing to our attention \cite[Chapter II, Section 1.d, Lemma 1.12]{Ny04}, which helped simplify the original argument for Theorem \ref{new-fiberbime-irredu-singular}. 
Our sincere thanks go to Professors Zhengyu Hu, Zhan Li, Ngaiming Mok, and Dr. Runze Zhang for valuable discussions and for answering our questions regarding the semi-stable reduction theorem, uniruledness, and canonical singularities.
We are also indebted to Professor Kang Zuo for his insightful comments and his interest in related topics. 
Finally, the first two authors would like to thank Professor Jih-Hsin Cheng for his kind assistance in arranging their visit to the Institute of Mathematics, Academia Sinica.}

\end{document}